\documentclass[11pt,reqno]{amsart}
\usepackage{amsmath, amssymb, amsthm}
\usepackage{url}
\usepackage{hyperref}
\usepackage{tikz}
\usepackage{ytableau}
\usepackage[utf8]{inputenc}
\usepackage{xcolor}
\usepackage{tikz-cd}
\usepackage{hyperref}
\usepackage{nccmath}

\usepackage{youngtab}

\usetikzlibrary{decorations.pathreplacing}

\usetikzlibrary{decorations.pathreplacing}

\renewcommand{\a}{\alpha}

\renewcommand{\(}{\left\(}
\renewcommand{\)}{\right\)}
\renewcommand{\[}{\left\[}
\renewcommand{\]}{\right\]}

\numberwithin{equation}{section}
\theoremstyle{plain}
\newtheorem{theorem}{Theorem}[section]
\newtheorem{lemma}[theorem]{Lemma}

\newtheorem{conjecture}[theorem]{Conjecture}

\newtheorem{corollary}[theorem]{Corollary}

\def\wsR{{\hat{\mathscr{L}}}}
\usepackage{mathrsfs}

\makeatletter
\def\proof{\@ifnextchar[{\@oproof}{\@nproof}}
\def\@oproof[#1][#2]{\trivlist\item[\hskip\labelsep\textit{#2 Proof of\
		#1.}~]\ignorespaces}
\def\@nproof{\trivlist\item[\hskip\labelsep\textit{Proof.}~]\ignorespaces}

\makeatother

\makeatletter
\@namedef{subjclassname@2020}{%
	\textup{2020} Mathematics Subject Classification}
\makeatother

\begin{document}
	
	\title[]{Asymptotic expansion for the Fourier coefficients associated with the inverse of the modular discriminant function $\Delta$} 
	
	\author{Gargi Mukherjee}
	\address{School of Mathematical Sciences,
		National Institute of Science Education and Research, Bhubaneswar, An OCC of Homi Bhabha National Institute,  Via- Jatni, Khurda, Odisha- 752050, India}
	\email{gargi@niser.ac.in}
	
	\maketitle
	
	\begin{abstract}
There have been a plethora of investigations carried out in studying inequalities for the Fourier coefficients of weakly holomorphic modular forms, for example, on the partition function. Recently, Bringmann, Kane, Rolen, and Tripp studied asymptotics for the $k$-colored partition function and more generally, for the fractional partitions arising from the Nekrasov-Okounkov formula which in turn allowed them to prove generalized multiplicative inequalities. Motivated by their idea to find interesting inequalities for the $k$-colored partition functions, denoted by $p_k(n)$, in this paper, we prove a family of inequalities for the $p_{24}(n)$. The main aim of this paper is to study the asymptotic expansion with an effective estimate for the error bound regarding the Fourier coefficients of the modular form $1/\Delta$ (up to a constant $c$ and a power of $q$), where $\Delta$ is the modular discriminant function and a well-known combinatorial interpretation for the associated coefficient sequence is called $24$-colored partitions, denoted by $p_{24}(n)$. Consequently, we show that $p_{24}(n)$ satisfies $2$-$\log$-concavity, Tur\'{a}n inequality of order $3$, and Laguerre inequalities of order $m$ with $2\le m\le 8$ eventually. Our method of estimations for the error term of the asymptotic expansion for $p_{24}(n)$ can be adapted in a more general paradigm where the Fourier coefficients of a certain class of Dedekind-eta quotients which are essentially a modular form of negative weight, admit a Rademacher type exact formula involving the $I$-Bessel function of positive order.   
	\end{abstract}

	\section{Introduction}\label{intro}
Study on the growth of the partition function $p(n)$ begun with the fundamental work of Hardy and Ramanujan \cite{HR} where they proved a divergent series formula for $p(n)$ and subsequently, extracting out the first term, we get the following asymptotic major term of $p(n)$
$$p(n)\sim \frac{e^{\pi\sqrt{\frac{2n}{3}}}}{4n\sqrt{3}}\ \ \text{as}\ \ n\rightarrow\infty.$$
Hardy and Ramanujan's proof fundamentally relies on the modular behavior of the Dedekind-eta function $\eta(\tau)=e^{\frac{\pi i\tau}{12}}\prod_{n\ge1}(1-e^{2\pi i\tau n})$ with $\tau\in \mathbb{H}$ 
and developed the widely celebrated tool called Circle Method. Later Rademacher \cite{R} gave a convergent series formula for $p(n)$ by perfecting the Circle Method of Hardy and Ramanujan. Since Rademacher's discovery, numerous research works have been done in this regime. Among many other investigations on analytic behavior of the sequences related to $p(n)$, perhaps the most notable was on the coefficients whose generating function is of the form $\eta(\tau)^{c(n)}$. For $c(n)=24$, we have the famous modular discriminant function $\Delta(\tau)$ and its coefficients are well-known as Ramanujan's tau function $\tau(n)$ studied by Ramanujan. Ramanujan's conjectures on the analytic behavior of $\tau(n)$ was confirmed by Deligne \cite{D} in his proof of Weil's conjectures. Ramanujan \cite{Rama} also studied the arithmetic properties of the coefficients $\eta(\tau)^{c(n)}$ with $c(n)=-k$ and $k\in \mathbb{N}$, known as the $k$-colored partitions $p_k(n)$ and for a more detailed study, we refer the reader to \cite{Atkin,Gordon}. For $k\in \mathbb{N}$,
$$\eta(\tau)^{-k}=q^{-\frac{k}{24}}\prod_{n=1}^{\infty}\frac{1}{(1-q^n)^{k}}=:q^{-\frac{k}{24}}\sum_{n=0}^{\infty}p_{k}(n)q^n,$$
and in combinatorial language, the coefficients $p_{k}(n)$ counts the total number of partitions of $n$ where parts are allowed to come up with $k$ colors. Thanks to the Nekrasov-Okounkov formula \cite{NO}, we can extend the power parameter $c(n)$ from discrete to a continuous one $\a\in \mathbb{R}$ which has applications in gauge theory and related random partitions. Iskander, Jain, and Talvola \cite{iskander2020} gave a Rademacher type exact formula for $p_{\a}(n)$ (with $\a$ being both discrete and continuous). In a more general set up, if we allow $c(n)$ to be a function depending on $n$, in particular, for $c(n)=n$, the coefficients are known as Plane partitions, denoted by $\text{PL}(n)$, first studied in a large extend by MacMahon \cite{M}. Wright \cite{Wr} provided an asymptotic formula for $\text{PL}(n)$ and among other improvements of Wright's formula, the most significant one was by Ono, Pujahari, and Rolen \cite[Theorem 1.3]{OPR} where they bound explicitly the error terms for the asymptotic expansion of $\text{PL}(n)$. Very recently, Bridges, Brindle, Bringmann, and Franke \cite{BBBF} considered a more broader set up to get asymptotic expansion for coefficients of the infinite generating function of the form $\prod_{n\ge 1}\frac{1}{(1-q^n)^{c(n)}}$ subject to certain conditions satisfied by $c(n)$.

Among several applications of the Hardy-Ramanujan-Rademacher formula along with Lehmer's estimation \cite{L} of error terms for $p(n)$, combinatorial inequalities for $p(n)$ are the ones which lie at the interface of number theory and combinatorics. A sequence $(a_n)_{n\ge 0}$ is called log-concave if $a^2_n-a_{n-1}a_{n+1}\ge 0$ for $n\ge 1$. In other words, the Jensen polynomial 
$$J^{d,n}_a(x):=\sum_{j=0}^{d}\binom{d}{j}a_{n+j}x^j$$
is hyperbolic\footnote{A polynomial with real coefficients is called hyperbolic if it has all real roots.} for $d=2$ and $n\ge 1$. Nicolas \cite{N} and later DeSalvo and Pak \cite{DP} proved that $(p(n))_{n\ge 26}$ is log-concave. A sequence $(a_n)_{n\ge 0}$ satisfies the Tur\'{a}n inequality of order $3$ if for all $n\ge 1$,
\begin{equation*}
4\left(a^2_n-a_{n-1}a_{n+1}\right)\left(a^2_{n+1}-a^2_{n}a_{n+2}\right)\ge \left(a_{n}a_{n+1}-a_{n-1}a_{n+2}\right)^2.\\
\end{equation*}
 Chen, Jia, and Wang proved that $(p(n))_{n\ge 95}$ satisfies the Tur\'{a}n inequality of order $3$; i.e., $J^{3,n}_p(x)$ has all real roots for all $n\ge 95$. In their seminal paper on hyperbolicity of Jensen polynomials for the Riemann zeta function $\zeta(s)$ at its point of symmetry, Griffin, Ono, Rolen, and Zagier \cite{GORZ} showed that the Jensen polynomials $J^{d,n}_p(x)$ is eventually hyperbolic that settled the conjecture of Chen, Jia, and Wang \cite[Conjecture 1.5]{CJW}. More generally, $(a_n)_{n\ge 0}$ is said to be $r$-log-concave if 
 \begin{align*}
 \wsR\{a_n\}_{n\geq N}, \wsR^2\{a_n\}_{n\geq N}, \ldots, \wsR^r\{a_n\}_{n\geq N}
 \end{align*}
 are all non-negative sequences, where
 $$\wsR\{a_n\}_{n\geq 0}=\{a_{n+1}^2-a_na_{n+2}\}_{n\geq0}\quad\text{and}
 \quad
 \wsR^k\{a_n\}_{n\geq 0}=\wsR\left(\wsR^{k-1}\{a_n\}_{n\geq 0}\right).$$
 Hou and Zhang \cite{Hou-Zhang-2019} proved that $(p(n))_{n \geq 221}$ is $2$-$\log$-concave. Dou and Wang \cite{DW} proved that $(p(n))_{n\ge N(m)}$ satisfies the Laguerre inequalities of order $1\le m\le 9$, where $(a_n)_{n\ge 0}$ is said to satisfy the Laguerre inequalities of order $m$ if 
 $$L_m(a_n):=\frac{1}{2}\sum_{k=0}^{2m}(-1)^{m+k}\binom{2n}{k}a(n+k)a(2m-k+n)\ge 0.$$
 Bringmann, Kane, Rolen, and Tripp \cite{KBLZ} proved the Bessenrodt-Ono type multiplicative inequalities, a generalization of the log-concavity for $p_{\a}(n)$ and settled the conjectures of Chern-Fu-Tang \cite[Conjecture 5.3]{CFT} and Heim-Neuhauser \cite{HN}. Ono, Pujahari, and Rolen \cite{OPR} settled the conjecture of Heim, Neuhauser and Tr\"{o}ger \cite[Conjecture 1]{HNT} by proving that $\text{PL}(n)$ is log-concave for all $n\ge 12$. 
 
In this paper, motivated by one of the follow-up ideas stated by Bringmann, Kane, Rolen, and Tripp in \cite[Section 5]{KBLZ} on finding other type of inequalities for $p_k(n)$, here we in particular show that $p_{24}(n)$ satisfies the $2$-log-concavity, Tur\'{a}n inequality of order $3$, and Laguerre inequalities of order $2\le m \le 8$ eventually (cf. Corollaries \ref{cor1}-\ref{cor3}). In order to prove these inequalities, our main goal is to determine the error bound and the coefficients in the asymptotic expansion of $p_{24}(n)$ in Theorem \ref{thm1.1} stated below.
	 
Here and throughout the rest, we follow the notation $f(n)=O_{\le C}\left(g(n)\right)$ to state $\left|f(x)\right|\le C\cdot g(x)$ for a positive function $g$ and the domain for $x$ will be specified accordingly. 
	 \begin{theorem}\label{thm1.1}
	 	For all $N \in \mathbb{N}$ and $n \ge n(N)$,
	 	\begin{equation*}
	 		p_{24}(n)=\frac{e^{4\pi \sqrt{n}}}{\sqrt{2}\cdot n^{\frac{27}{4}}}\left(\sum_{m=0}^{N}\frac{\widetilde{B}_m}{n^{\frac{m}{2}}}+O_{\le E_N}\left(n^{-\frac{N+1}{2}}\right)\right),
	 	\end{equation*}
 	where $n(N)$, $\widetilde{B}_m$ and $E_N$ are given in \eqref{finalcutoffdef}, \eqref{finalcoeffdef}, and \eqref{finalerrordef} respectively.
	 \end{theorem}
As a direct consequence of Theorem \ref{thm1.1}, we have the following inequalities satisfied by $p_{24}(n)$ for sufficiently large $n$.
\begin{corollary}\label{cor1}
$p_{24}(n)$ satisfies the Tur\'{a}n inequality of order $3$ for sufficiently large $n$.
\end{corollary}
\begin{proof}
Taking $N=9$ in Theorem \ref{thm1.1}, we get
\begin{align*}
&4\left(p_{24}(n)^2-p_{24}(n-1)p_{24}(n+1)\right)\left(p_{24}(n+1)^2-p_{24}(n)p_{24}(n+2)\right)\\
&\hspace{5 cm}-\left(p_{24}(n)p_{24}(n+1)-p_{24}(n-1)p_{24}(n+2)\right)^2\\
&\hspace{5 cm}=\left(\frac{e^{4\pi \sqrt{n}}}{\sqrt{2}\cdot n^{\frac{27}{4}}}\right)^4\left(\frac{4\pi^3}{n^{\frac 92}}-\frac{\pi^2\left(1683+64\pi^2\right)}{n^5}+O\left(\frac{1}{n^{\frac{11}{2}}}\right)\right),
\end{align*}
which implies that $p_{24}(n)$ satisfies the Tur\'{a}n inequality of order $3$ eventually. 
\end{proof} 
\begin{corollary}\label{cor2}
$p_{24}(n)$ is $2$-$\log$-concave for sufficiently large $n$.	
\end{corollary}
\begin{proof}
Taking $N=9$ in Theorem \ref{thm1.1}, it follows that
\begin{equation*}
\wsR^2\left(p_{24}(n)\right)=\left(\frac{e^{4\pi \sqrt{n}}}{\sqrt{2}\cdot n^{\frac{27}{4}}}\right)^4\left(\frac{2\pi^3}{n^{\frac 92}}-\frac{\pi^2\left(843+64\pi^2\right)}{n^5}+O\left(\frac{1}{n^{\frac{11}{2}}}\right)\right),
\end{equation*}
which concludes the proof.
\end{proof}
\begin{corollary}\label{cor3}
$p_{24}(n)$ satisfies the Laguerre inequalities of order $2\le m\le 8$ eventually.	
\end{corollary}
\begin{proof}
For $2\le m\le 8$, choosing $N=3m$ in Theorem \ref{thm1.1}, we see that
\begin{equation*}
L_m\left(p_{24}(n)\right)=\left(\frac{e^{4\pi \sqrt{n}}}{\sqrt{2}\cdot n^{\frac{27}{4}}}\right)^2\left(\frac{(2\pi)^{m}(2m-1)!!}{2\cdot n^{\frac{3m}{2}}}+O\left(\frac{1}{n^{\frac{3m+1}{2}}}\right)\right),
\end{equation*}
which finishes the proof.
\end{proof}
Based on Corollary \ref{cor3}, here we propose the following conjecture on the asymptotic growth of $L_m(p_{24}(n))$.
\begin{conjecture}\label{conjLagpoly}
For $m\in \mathbb{N}$,
\begin{equation*}
L_m(p_{24}(n))\sim \frac{(2\pi)^m (2m-1)!!}{4\cdot n^{\frac{3m}{2}}}\frac{e^{8\pi\sqrt{n}}}{n^{\frac{27}{2}}}\ \ \text{as}\ \ n\rightarrow\infty.
\end{equation*}	
\end{conjecture}

The rest of this paper is organized as follows. In Section \ref{sec0}, we will set a premise by extracting out a Poincar\'{e} type asymptotic expansion for $p_{24}(n)$ from the exact formula of $p_{24}(n)$ (cf. Theorem \ref{thm2.1}) using the error analysis of the $I$-Bessel function given in Theorem \ref{thmKB}. In Section \ref{sec1}, we will first state a few preparatory lemmas which will assist to prove Theorem \ref{thm1.1}. Finally, in Section \ref{conclusion}, we will discuss about the prospects of Theorem \ref{thm1.1} in both directions; advantages and obstructions from the perspective of its applications to prove inequalities explicitly.
 
\section{Set up}\label{sec0}
In this section, starting from the Rademacher type exact formula for $p_{\a}(n)$ (by setting $\a=24$) in stated Theorem \ref{thm2.1} and following the error bound estimations of $I$-Bessel function given in Theorem \ref{thmKB}, we shall obtain an asymptotic expansion of $p_{24}(n)$ delineated in \eqref{eqn2.4}.
 
\begin{theorem}\cite[Theorem 1.1]{iskander2020}\label{thm2.1}
	For all $\alpha \in \mathbb{R}_{>0}$   and $n \geq \frac{\alpha}{24}$, we have 
	\begin{align}\nonumber
	p_{\alpha}(n)&= 2 \pi \left(n-\frac{\alpha}{24}\right)^{-\frac{\alpha}{4}-\frac{1}{2}} \sum_{m=0}^{\left\lfloor\frac{\alpha}{24}\right\rfloor} \left(\frac{\alpha}{24}-m\right)^{\frac{\alpha}{4}+\frac{1}{2}} p_{\alpha}(m)\\\label{thm2.1eqn}
	&\hspace*{4cm}   \times \sum_{k \geq 1}^{}\frac{A_{k,\alpha}(n,m)}{k} I_{\frac{\alpha}{2}+1}\left(\frac{4\pi}{k} \sqrt{\left(\frac{\alpha}{24}-m\right)\left(n-\frac{\alpha}{24}\right)}\right).
	\end{align} 
\end{theorem}
The $I$-\textit{Bessel function} is defined by $$I_{\alpha}(x):=\sum_{m=0}^{\infty}\frac{\left(\frac{x}{2}\right)^{2m+\alpha}}{m!\ \Gamma(m+\alpha+1)}\ \ \text{with}\ \left(x,\alpha\right) \in\mathbb{R}^2,$$
and following the integral representation (see \cite[page 172]{Wat}), one can derive the well-known asymptotic expansion of $I$-Bessel function\footnote{For a more detailed study on $I$-Bessel function, we refer to Watson's magnum opus \cite{Wat}.} 

Extracting out the first term from \eqref{thm2.1eqn} and using the asymptotic growth of $I_{\a}(x)$
$$I_{\a}(x)\sim \frac{e^x}{\sqrt{2\pi x}}\ \ \text{as}\ x\rightarrow\infty,$$
with $\lambda_{\a}(n):=\sqrt{\frac{24n}{\a}-1}$, we can retrieve the asymptotic main term for $p_{\alpha}(n)$ which reads
$$p_{\alpha}(n) \sim \sqrt{\frac{12}{\alpha}} \cdot \frac{e^{\frac{\pi \alpha}{6}\lambda_{\alpha}(n)}}{\lambda_{\alpha}(n)^{\frac{\alpha+3}{2}}}\ \ \text{as}\ n\rightarrow \infty.$$
Since $p_{24}(0):=1$ and for $(\a,m)\in\mathbb{N}\times\mathbb{N}_0$, $A_{1,\a}(n,m)=1$, setting $\alpha=24$ in Theorem \ref{thm2.1} we get for all $n \geq 1$, \begin{equation}\label{eqn2.1}
p_{24}(n)=2 \pi (n-1)^{-\frac{13}{2}}\sum_{k\geq1}^{}\frac{A_{k,24}(n,0)}{k}I_{13}\Bigl(\frac{4\pi}{k}\sqrt{n-1}\Bigr)=:M(n)+R(n),
\end{equation}
where
\begin{equation}\label{MnRndef}
M(n):=\frac{2\pi}{(n-1)^{\frac{13}{2}}}I_{13}\left(4\pi \sqrt{n-1}\right)\ \text{and}\ R(n):=\frac{2\pi}{(n-1)^{\frac{13}{2}}}\sum_{k\geq 2}\frac{A_{k,24}(n,0)}{k}I_{13}\left(\frac{4\pi}{k}\sqrt{n-1}\right).
\end{equation}  
Now applying the trivial estimate for the \textit{Kloosterman sum} $A_{k,\a}(n,m)$; i.e., $\left| A_{k,\alpha}(n,m)\right|\leq k$ to \eqref{MnRndef}, we get
\begin{equation}\label{Rnbound}
\left
| R(n)\right 
| \leq \frac{2\pi}{(n-1)^{\frac{13}{2}}} \sum_{k\geq 2}^{} \left 
| \frac{A_{k,24}(n,0)}{k} I_{13}\Bigl(\frac{4\pi}{k}\sqrt{n-1}\Bigr) \right |\le \frac{2\pi}{(n-1)^{\frac{13}{2}}}\sum_{k\geq 2}^{}I_{13}\Bigl(\frac{4\pi}{k}\sqrt{n-1}\Bigr).
\end{equation}
Next, we refine the upper bound of $\left|R(n)\right|$ in \eqref{Rnbound}. To do so, first, following the sketch done in \cite[Proof of Theorem 1.2]{DJJ} and using the monotonicity of $I_{13}(x)$ on $(0,\infty)$, we have for any $N \in \mathbb{N}$ 
\begingroup
\allowdisplaybreaks
\begin{align}\nonumber
&\sum_{k \geq N+1} I_{13}\left(\frac{y}{k}\right)= \sum_{k \geq N+1} \sum_{m \geq 0} \frac{1}{m!(m+13)!} \left(\frac{y}{2k}\right)^{2m+13}\le \int_{N}^{\infty} \sum_{m \geq 0} \frac{1}{m!(m+13)!} \left(\frac{y}{2t}\right)^{2m+13} dt\\\nonumber
&\le \sum_{m \geq 0} \frac{1}{m!(m+13)!} \int_{N}^{\infty} \left(\frac{y}{2t}\right)^{2m+13} dt=N \sum_{m \geq 0} \frac{1}{m!(m+13)!(2m+12)} \left(\frac{y}{2N}\right)^{2m+13}\\\nonumber
&\le N \sum_{m \geq 0} \frac{1}{(m+1)!(m+13)!} \left(\frac{y}{2N}\right)^{2m+13}\ \ \left(\text{since}\ (2m+12) \cdot m! \geq (m+1)!\right)\\\label{eqn2.2}
&=N \sum_{m \geq 1}\frac{1}{m!(m+12)!} \left(\frac{y}{2N}\right)^{2m+11}\le \frac{2 N^2}{y}\sum_{m \geq 0} \frac{1}{m!(m+12)!} \left(\frac{y}{2N}\right)^{2m+12}=\frac{2 N^2}{y} I_{12}\left(\frac{y}{N}\right).
\end{align}
\endgroup
Next, splitting the sum in \eqref{Rnbound} for $k=2$ and $k\ge 3$, it follows that
\begin{equation}\label{Rnbound1}
\left|R(n)\right| \leq \frac{2\pi}{(n-1)^{\frac{13}{2}}}I_{13}\left(2 \pi\sqrt{n-1}\right)+ \frac{2\pi}{(n-1)^{\frac{13}{2}}}\sum_{k\geq 3}I_{13}\left(\frac{4\pi}{k}\sqrt{n-1}\right).
\end{equation}
Applying \eqref{eqn2.2} with $(N,y)\mapsto \left(2,4\pi\sqrt{n-1}\right)$, we obtain
\begin{equation*}
\sum_{k\geq 3}I_{13}\left(\frac{4\pi}{k}\sqrt{n-1}\right) \leq \frac{8}{4\pi\sqrt{n-1}} I_{12}\left(2\pi\sqrt{n-1}\right),
\end{equation*}
and consequently applying \cite[Lemma 2.2 (1)]{KBLZ} with $(\kappa, x)\mapsto \left(12,2\pi\sqrt{n-1}\right)$, we get for all $n\ge 2$,
\begin{equation}\label{Rnbound21}
\frac{2\pi}{(n-1)^{\frac{13}{2}}}\sum_{k\geq 3}I_{13}\left(\frac{4\pi}{k}\sqrt{n-1}\right) \le \frac{4}{\pi(n-1)^{\frac{29}{4}}}e^{2\pi\sqrt{n-1}}.
\end{equation}
Similarly, for the remaining term in \eqref{Rnbound1}, using \cite[Lemma 2.2 (1)]{KBLZ} with $(\kappa, x)\mapsto \left(13,2\pi\sqrt{n-1}\right)$, we get for all $n\ge 2$,
\begin{equation}\label{Rnbound22}
I_{13}\left(2 \pi\sqrt{n-1}\right)\le \sqrt{\frac{2}{\pi}}\frac{e^{2\pi\sqrt{n-1}}}{\left(2\pi\sqrt{n-1}\right)^{\frac{1}{2}}}.
\end{equation}
Applying \eqref{Rnbound21} and \eqref{Rnbound22} to \eqref{Rnbound1}, for all $n\ge 2$, we have
\begin{equation}\label{Rnboundfinal}
\left|R(n)\right| \leq 2^{29} \pi^{\frac{27}{2}} \frac{e^{\frac{x(n)}{2}}}{x(n)^{\frac{27}{2}}}\ \ \text{with}\ \ x(n):=4 \pi \sqrt{n-1}.
\end{equation}
Let $M(n)$ be as in \eqref{MnRndef} and set
\begin{equation}\label{Gndef}
	G(n):=\frac{2^{29}\pi^{\frac{27}{2}} \frac{e^{\frac{x(n)}{2}}}{x(n)^{\frac{27}{2}}}}{M(n)}=\frac{4 e^{\frac{x(n)}{2}}}{\sqrt{\pi x(n)}I_{13}(x(n))}.
\end{equation}
Before we move on to estimate $G(n)$, we record the error estimations of the $I$-\textit{Bessel function} as in \cite{KB} in the present context; i.e., by setting $\nu=13$-the order of the $I$-\textit{Bessel function}. Here and throughout the rest of the paper, $(b)_m$ stands for the usual rising factorial defined by
\begin{equation*}
(b)_m:=\begin{cases}
b(b+1)\cdots(b+m-1), &\quad \text{if}\ m\ge 1,\\
1, &\quad \text{if}\ m=0.
\end{cases}
\end{equation*}
\begin{theorem}\cite[Theorem 3.9]{KB}\label{thmKB}
Let $x\in\mathbb{R}_{>1}$ and $N\in\mathbb{N}$. Then
$$\left|\frac{\sqrt{2\pi x}}{e^x}I_{13}(x)-\sum_{m=0}^{N}\frac{(-1)^ma_m(13)}{x^m}\right|<\frac{E(13,N)\left|a_{N+1}(13)\right|}{x^{N+1}},$$
where 
$$a_m(13)=\frac{\binom{13-\frac{1}{2}}{m}(13+\frac{1}{2})_m}{2^m},$$
and
\begin{equation*}
E(13,N)=\\
\begin{cases}
1+\left(\frac{2}{\pi^3}\right)^{\frac 12}\sqrt{\frac{14}{12-N}}\sqrt{N+\frac{29}{2}}\left(\sqrt{\frac{1}{13-N}}-\sqrt{\frac{1}{14}}\right), &\quad \text{if}\ N\le 11,\\
1+\frac{\sqrt{29}}{\pi}\left(\sqrt{14}-1\right), &\quad \text{if}\ N=12,\\
1+\frac{2\sqrt{\pi}+3^4\cdot 5}{2\sqrt{\pi}}\sqrt{2N+29}+\frac{\sqrt{2N+29}\log(N+1)}{2\sqrt{\pi}}+\frac{3^4\cdot 5\sqrt{2N+29}\log(N+1)}{2\sqrt{\pi}(N+2)}, &\quad \text{if}\ N\ge 13.
\end{cases}
\end{equation*}
\end{theorem}
Define
\begin{equation*}
C^{*}(m) :=
\begin{cases}
1, &\quad \text{if}\ m=1,\\
6m \log m-m \log \log m, & \quad \text{if}\ m \geq 2.
\end{cases}
\end{equation*}
\begin{lemma}\label{lemma2.2} 
 Let $N\in\mathbb{N}$.	Then for all
 \begin{equation}\label{finalcutoffdef}
 n \geq \underset{N\ge 1}{\max}\left\{138, \left\lceil \left(\frac{C^{*}(N+2)}{4\pi}\right)^2+1\right\rceil\right\}=:n(N),
 \end{equation}
 we have 
 $$G(n)=O_{\le 1}\left(n^{-\frac{N+1}{2}}\right).$$
\end{lemma}
\begin{proof}
Applying Theorem \ref{thmKB} with $(x, N)\mapsto \left( x(n), 1\right)$, we get for $x(n)\in {\mathbb{R}}_{\geq 1}$, 
	\begin{equation}\label{lemma2.2eqn1}
	\left|\frac{I_{13}(x(n)) \sqrt{2 \pi x(n)}}{e^{x(n)}} - \sum_{m=0}^{1} \frac{(-1)^m a_m(13)}{x(n)^m}\right| < \frac{E(13,1)\left|a_{2}(13)\right|}{x(n)^{2}},
	\end{equation}
	which consequently implies that
	\begingroup
	\allowdisplaybreaks
	\begin{align}\nonumber
	I_{13}(x(n)) &\ge \frac{e^{x(n)}}{\sqrt{2 \pi x(n)}} \left(\sum_{m=0}^{1} \frac{(-1)^m a_m(13)}{x(n)^m} -\frac{E(13,1) \cdot a_{2}(13)}{x(n)^{2}}\right)\\\nonumber
	&\ge \frac{e^{x(n)}}{\sqrt{2 \pi x(n)}} \left(\sum_{m=0}^{1} \frac{(-1)^m a_m(13)}{x(n)^m} -\frac{2 \cdot a_{2}(13)}{x(n)^{2}}\right)\ \ \left(\text{as}\ E(13,1)<2\right)\\\label{lemma2.2eqn2}
	&= \frac{e^{x(n)}}{\sqrt{2 \pi x(n)}} \Biggl(1-\frac{ a_1(13)}{x(n)} -\frac{2 \cdot a_{2}(13)}{x(n)^{2}}\Biggr)\ge \frac{1}{10}\frac{e^{x(n)}}{\sqrt{2 \pi x(n)}},
	\end{align}
	\endgroup
	where in the last step, we used the fact that for $n \geq 138$,
	$$1-\frac{ a_1(13)}{x(n)} -\frac{2 \cdot a_{2}(13)}{x(n)^{2}}>\frac{1}{10}.$$
From \eqref{Gndef} and \eqref{lemma2.2eqn2}, it follows that for all $x(n)\ge \underset{N\ge 1}{\max}\left\{C^*(N+2),138\right\}$
\begin{equation}\label{lemma2.2eqn3}
|G(n)|<40\sqrt{2}e^{-\frac{x(n)}{2}}<\frac{40\sqrt{2}}{x(n)^{N+2}}\le\frac{1}{x(n)^{N+1}}\le n^{-\frac{N+1}{2}},
\end{equation}
where in the second step, we use the fact\footnote{We omit the proof as it is analogous to \cite[Lemma 2.1]{MZZ}} $e^{-\frac{y}{2}}<y^{-m}$ for all $y\ge C^*(m)$ and consequently substitute $(y,m)\mapsto (x(n),N+2)$. Finally noting that
$$x(n)\ge \underset{N\ge 1}{\max}\left\{C^*(N+2),138\right\}\ \text{implies}\  n\ge n(N),$$
we conclude the proof.
\end{proof}
Applying Theorem \ref{thmKB} with $x\mapsto x(n)$ and Lemma \ref{lemma2.2} to \eqref{eqn2.1}, we see that for all $n\ge n(N)$,
\begingroup
\allowdisplaybreaks
\begin{align}\nonumber
p_{24}(n)&=\frac{e^{4 \pi \sqrt{n}}}{\sqrt{2}\cdot n^{\frac{27}{4}}}\frac{e^{4 \pi \left(\sqrt{n-1}-\sqrt{n}\right)}}{(1-\frac{1}{n})^{\frac{27}{4}}}\left(\sum_{m=0}^{N}\frac{(-1)^{m}a_m(13)}{(4 \pi)^{m}(n-1)^{\frac m2}}+O_{\le E_{N,1}}\left(n^{-\frac{N+1}{2}}\right)\right)\\\label{eqn2.4}
&\hspace{8 cm}\cdot \left(1+O_{\le 1}\left(n^{-\frac{N+1}{2}}\right)\right),
\end{align}
\endgroup 
with 
$$E_{N,1}:=\frac{2 \cdot E(13,N)\left|a_{N+1}(13)\right|}{(4 \pi)^{N+1}}.$$
	\section{Proof of Theorem \ref{thm1.1}}\label{sec1}

In this section, we will present some preparatory lemmas (cf. Lemmas \ref{Glem1}-\ref{Glem5}) in order to prove Theorem \ref{thm1.1}. We will estimate the error bounds for the remainder terms of the Taylor expansions (truncated at $N$) of the following three factors from \eqref{eqn2.4}:
\begin{equation*}
\exp\left(4\pi\left(\sqrt{n-1}-\sqrt{n}\right)\right), \left(1-\frac 1n\right)^{-\a}\ \left(\text{with}\ \a\in\mathbb{R}_{\ge 1}\right), \text{and}\ \sum_{m=0}^{N}\frac{(-1)^{m}a_m(13)}{(4 \pi)^{m}(n-1)^{\frac m2}}.
\end{equation*}
Let us state in brevity about the estimations done in these Lemmas \ref{Glem1}-\ref{Glem5}. First we estimate an error bound for the remainder terms of the Taylor expansion of the exponential term in Lemma \ref{Glem1}. We follow closely the proof carried out in \cite[Lemma 4.6]{B1}. Before proceed to the lemma, let us state the following identity on convoluted binomial sums which will be helpful to prove Lemma \ref{Glem1}. Subsequently, after error bound estimations for the factor $\left(1-\frac 1n\right)^{-\a}$ in Lemma \ref{Glem2}, we combine Lemmas \ref{Glem1} and \ref{Glem2} (with $\a=\frac{27}{4}$) to obtain an estimate of the error terms for the convoluted coefficient sequence arises from the Taylor expansion of $\exp\left(4\pi\left(\sqrt{n-1}-\sqrt{n}\right)\right)\left(1-\frac 1n\right)^{-\frac{27}{4}}$ in Lemma \ref{Glem3}. Then we will carry out the error bound estimation for the factor $\sum_{m=0}^{N}\frac{(-1)^{m}a_m(13)}{(4 \pi)^{m}(n-1)^{\frac m2}}$ in Lemma \ref{Glem4} and consequently, in Lemma \ref{Glem5}, we merge Lemma \ref{Glem4} with the remaining error terms $O(n^{-\frac{N+1}{2}})$ in \eqref{eqn2.4}.

		\begin{lemma}\cite[Lemma 3.3]{KPRS}\label{lemma2.3}
			For $r,m\in {\mathbb{N}}_0$ with $r<2m$, we have 
			\begin{equation}\label{def1}
			\sum_{s=0}^{r}\binom{r}{s}\binom{\frac{s}{2}}{m}=
			\begin{cases}
			1, &\quad \text{if}\ r=m=0,\\
			(-1)^m \frac{r \cdot 2^r}{m \cdot 2^{2m}} \binom{2m-r-1}{m-r}, & \quad \text{otherwise}.
			\end{cases}
			\end{equation}
		\end{lemma} 
For $k\in \mathbb{N}_0$, define
\begin{equation}\label{Glem1def1}
A_1(k):=
\begin{cases}
1, &\quad \text{if}\ k=0,\\
\frac{(-1)^k\left(\frac{1}{2}-k\right)_{k+1}}{k} \displaystyle\sum_{\ell=1}^{k}\frac{(-1)^\ell (-k)_\ell}{(k+\ell)!}\frac{(4\pi)^{2\ell}}{(2\ell-1)!}, &\quad \text{if}\ k\ge 1,
\end{cases}
\end{equation}
and
\begin{equation}\label{Glem1def2}
A_2(k):=(-1)^{k+1}\left(\frac{1}{2}-k \right)_{k+1} \sum_{\ell=0}^{k}\frac{(-1)^\ell(-k)_\ell}{(\ell+k+1)!}\frac{(4\pi)^{2\ell+1}}{(2\ell)!}.
\end{equation}
		\begin{lemma}\label{Glem1}
			For $N \geq 1$ and $n \geq 2$, we have 
			\begin{equation*}
				e^{4 \pi (\sqrt{n-1}-\sqrt{n})}= \sum_{k=0}^{N}\frac{T_k}{n^{\frac{k}{2}}} + O_{\leq E_2}\left(n^{-\frac{N+1}{2}}\right)\ \ \text{with}\ \ E_2:=2 \sqrt{2 \pi}\cosh\left(4 \pi\right),
			\end{equation*}
	and $T_{2k}=A_1(k)$, $T_{2k+1}=A_2(k)$ for all $k\in \mathbb{N}_0$. 
		\end{lemma}
	\begin{proof}
		By Taylor expansion, we have
		\begingroup
		\allowdisplaybreaks
		\begin{align}\nonumber
		e^{4 \pi(\sqrt{n-1}-\sqrt{n})}&=e^{4 \pi \sqrt{n}\left(\sqrt{1-\frac{1}{n}}-1\right)}=\sum_{r=0}^{\infty}\frac{(4 \pi \sqrt{n})^r}{r!} \left(\sqrt{1-\frac{1}{n}}-1\right)^r\\\nonumber
		&=\sum_{r=0}^{\infty}\frac{(4 \pi \sqrt{n})^r}{r!} \sum_{s=0}^{r}(-1)^{r+s} \binom{r}{s} \sum_{m=0}^{\infty}\binom{\frac{s}{2}}{m}\frac{(-1)^m}{n^m}\\\nonumber
		&=\sum_{r,m\geq 0}^{\infty}\sum_{s=0}^{r}\frac{(4 \pi)^r}{r!}(-1)^{r+s+m} \binom{r}{s} \binom{\frac{s}{2}}{m}n^{-\frac{2m-r}{2}}\\\label{eqn2.5}
		&=\sum_{m=0}^{\infty}\sum_{r=0}^{2m}\sum_{s=0}^{r}\frac{(4 \pi)^r}{r!}(-1)^{r+s+m} \binom{r}{s} \binom{\frac{s}{2}}{m}n^{-\frac{2m-r}{2}},
		\end{align}
		\endgroup
where in the last step, we truncate the range $0\le r< \infty$	to $0\le r\le 2m$ because with $z:=n^{-\frac 12}$, $T\left(\frac{1}{z^2}\right):=\exp\left(\frac{4\pi}{z}\left(\sqrt{1-z^2}-1\right)\right)$ is analytic in neighborhood of $0$ and so $T(n)$ admits a Taylor expansion of the form $\sum_{j\ge 0}c_jn^{-\frac j2}$ which in turn yields that $2m-r\ge 0$; i.e., $0\le r\le 2m$. 	

Define $D:=\{(r,s,m)\in {\mathbb{N}}_0 : 0 \leq s \leq r \}$ and for $k \in {\mathbb{N}}_0$, $D_k:=\{(r,s,m)\in {\mathbb{N}}_0 : 2m-r=k \}$. Also for $t=(r,s,m)\in D$, define $a_t:=\frac{(4 \pi)^r}{r!}(-1)^{r+s+m} \binom{r}{s} \binom{\frac{s}{2}}{m}$, $d_t:=2m-r$. Thus \eqref{eqn2.5}) can be rewritten as follows
\begin{equation}\label{eqn2.5a}
e^{4 \pi(\sqrt{n-1}-\sqrt{n})}=\sum_{t\in(r,s,m)\in D}^{}\frac{a_t}{n^{\frac{d_t}{2}}}=\sum_{k\geq 0}\sum_{t\in D_k}\frac{a_t}{n^{\frac{k}{2}}}=\sum_{k\geq 0}\sum_{t\in D_{2k}}\frac{a_t}{n^{k}}+\sum_{k\geq 0}\sum_{t\in D_{2k+1}}\frac{a_t}{n^{k+\frac{1}{2}}}.
\end{equation}
We note that
\begingroup
\allowdisplaybreaks
\begin{align*}
D_{2k}=\left\{(r,s,m)\in D : r-2m=-2k\right\}&=\left\{(r,s,m)\in D: r\equiv 0 \ (\text{mod}\hspace*{0.2cm}2), \ r-2m=-2k\right\}\\
&=\{(2\ell,s,\ell+k)\in \mathbb{N}_0^3: 0\leq s \leq 2\ell\}.
\end{align*}
\endgroup
Therefore, we have 
\begingroup
\allowdisplaybreaks
\begin{align}\nonumber
\sum_{k\geq0}^{\infty}\sum_{t\in D_{2k}}^{}\frac{a_t}{n^{-k}}&=\sum_{\ell,k \geq 0}^{\infty}\sum_{s=0}^{2\ell} \frac{(4 \pi)^{2\ell}}{(2\ell)!}(-1)^{\ell+s+k}\binom{2\ell}{s}\binom{\frac{s}{2}}{\ell+k} \frac{1}{n^k}\\\label{eqn2.6}
&=\sum_{k \geq 0}^{\infty} \left((-1)^k\sum_{\ell \geq 0}^{\infty}(-1)^\ell \frac{(4 \pi)^{2\ell}}{(2\ell)!} \sum_{s=0}^{2\ell}(-1)^s \binom{2\ell}{s} \binom{\frac{s}{2}}{\ell+k}\right).
\end{align}
\endgroup 
By Lemma \ref{lemma2.3}, we have the following 
\begin{equation*}
	\sum_{s=0}^{2\ell}(-1)^s \binom{2\ell}{s} \binom{\frac{s}{2}}{\ell+k}= \begin{cases}
		1,&\quad\text{if} \  \ell=k=0,\\
		0, &\quad\text{if} \ \ell>k,\\ 
		\frac{2\ell \left(\frac{1}{2}-k\right)_{k+1} (-k)_\ell}{k(k+\ell)!}, &\quad\text{if} \ 1 \leq \ell \leq k,
	\end{cases}
\end{equation*}
and therefore, we get
\begingroup
\allowdisplaybreaks
\begin{align}\nonumber
\sum_{\ell \geq 0}(-1)^\ell \frac{(16 \pi^2)^\ell}{(2\ell)!} \sum_{s=0}^{2\ell}(-1)^s \binom{2\ell}{s} \binom{\frac{s}{2}}{\ell+k}&=\sum_{\ell= 1}^{k}(-1)^\ell\frac{(16 \pi^2)^\ell}{(2\ell)!}\frac{2\ell \left(\frac{1}{2}-k\right)_{k+1} (-k)_\ell}{k(k+\ell)!}\\\label{eqn2.6a}
&=\frac{ \left(\frac{1}{2}-k\right)_{k+1}}{k} \sum_{\ell=1}^{k}(-1)^\ell \frac{(16 \pi^2)^\ell}{(2\ell-1)!} 	\frac{ (-k)_\ell}{(k+\ell)!}.
\end{align}
\endgroup 
Applying \eqref{eqn2.6a} to \eqref{eqn2.6} and following \eqref{Glem1def1}, we get for $k \in \mathbb{N}$, 
\begingroup
\allowdisplaybreaks
\begin{equation}\label{eqn2.7}
	\sum_{k\geq 0}^{\infty}\sum_{t\in D_{2k}}\frac{a_t}{n^{k}}= 1+ \sum_{k \geq 1}(-1)^k\left(\frac{\left(\frac{1}{2}-k\right)_{k+1}}{k} \sum_{\ell=1}^{k}(-1)^\ell \frac{(16 \pi^2)^\ell}{(2\ell-1)!} 	\frac{ (-k)_\ell}{(k+\ell)!}\right)= \sum_{k \geq 0}\frac{A_1(k)}{n^k}.	
\end{equation}
\endgroup 
Analogously, $D_{2k+1}=\left\{(2\ell+1,s,\ell+k+1) \in {{\mathbb{N}}_0}^3: 0 \leq s \leq 2\ell+1\right\}$ and applying Lemma \ref{lemma2.3} followed by \eqref{Glem1def2} we obtain
\begin{equation}\label{eqn2.8}
\sum_{k\geq 0}^{\infty}\sum_{t\in D_{2k+1}}\frac{a_t}{n^{-k-\frac{1}{2}}}=-\sum_{k\geq 0}^{\infty} \left(\left(\frac{1}{2}-k \right)_{k+1} \sum_{\ell=0}^{k}\frac{(-1)^\ell (-k)_\ell (4 \pi)^{2\ell+1}}{(2\ell)! (\ell+k+1)!}\right)\frac{(-1)^k}{n^{k+\frac{1}{2}}}=\sum_{k \geq 0}^{\infty}\frac{A_2(k)}{n^{k+\frac{1}{2}}}.
\end{equation}
Applying \eqref{eqn2.7} and \eqref{eqn2.8} to \eqref{eqn2.5a}, we get
\begin{equation}\label{eqn2.9}
	e^{4 \pi (\sqrt{n-1}-\sqrt{n})}=\sum_{k \geq 0}^{N}\frac{T_k}{n^{\frac{k}{2}}}+\sum_{k \geq N+1}^{\infty}\frac{T_k}{n^{\frac{k}{2}}}\ \ \text{with}\ T_{2k}=A_1(k)\ \text{and}\ T_{2k+1}=A_2(k).
\end{equation}
Next, we estimate an upper bound for the remainder infinite series in \eqref{eqn2.9} and for that, need to estimate for $|T_k|$ with $k\in \mathbb{N}$. Observe that for $n\in \mathbb{N}$,
\begin{equation}\label{binomestim}
\binom{2n}{n}\le\frac{4^n}{\sqrt{\pi n}}.
\end{equation}
First, we estimate an upper bound for $A_1(k)$ as
\begingroup
\allowdisplaybreaks
\begin{align}\nonumber
\left|A_1(k)\right|&\le \frac{\left|\left(\frac{1}{2}-k\right)_{k+1}\right|}{k}\sum_{\ell=1}^{k}\frac{(-1)^\ell (-k)_\ell}{(k+\ell)!}\frac{(4\pi)^{2\ell}}{(2\ell-1)!}\ \ \left(\text{by}\ \eqref{Glem1def1}\right)\\\nonumber
&=\frac{\binom{2k}{k}}{2k\cdot 4^k}\sum_{\ell=1}^{k}\left(\prod_{j=0}^{\ell-1}\frac{k-j}{k+j-1}\right)\frac{(4\pi)^{2\ell}}{(2\ell-1)!}\le \frac{1}{2\sqrt{\pi}k^{\frac 32}}\sum_{\ell=1}^{k}\frac{(4\pi)^{2\ell}}{(2\ell-1)!}\ \ \left(\text{by} \eqref{binomestim}\right)\\\label{eqn2.10}
&\le \frac{2\sqrt{\pi}}{k^{\frac 32}}\sinh\left(4\pi\right).
\end{align} 
\endgroup
Similarly, following \eqref{Glem1def2} and applying \eqref{binomestim}, for all $k\in \mathbb{N}$, we get
\begin{equation}\label{eqn2.11}
\left|A_2(k)\right|\le \frac{2 \sqrt{\pi}}{k^{\frac{3}{2}}} \cosh\left(4\pi\right).
\end{equation}
Combining \eqref{eqn2.9} and \eqref{eqn2.10}, it follows that
\begin{equation}\label{Akfinalbound}
\left|T_k\right| \underset{k\in \mathbb{N}}{\le} \left\{\frac{2\sqrt{\pi}}{k^{\frac 32}}\sinh\left(4\pi\right),\frac{2 \sqrt{\pi}}{k^{\frac{3}{2}}} \cosh\left(4\pi\right) \right\}=\frac{2 \sqrt{\pi}}{k^{\frac{3}{2}}} \cosh\left(4\pi\right),
\end{equation}
and therefore, applying \eqref{Akfinalbound}, we get for all $N\ge 1$ and $n\ge 2$,
\begingroup
\allowdisplaybreaks
\begin{align}\nonumber 
	\left|\sum_{k \geq N+1}\frac{T_k}{n^{\frac{k}{2}}}\right|&\le \sum_{k \geq N+1}\frac{\left|T_k\right|}{n^{\frac{k}{2}}}\le 2 \sqrt{\pi} \cosh\left(4 \pi\right)\sum_{k \geq N+1} \frac{1}{k^{\frac{3}{2}}\cdot n^{\frac{k}{2}}}\le \frac{2 \sqrt{\pi}\cosh\left(4 \pi\right)}{(N+1)^{\frac{3}{2}}}\sum_{k \geq N+1}\frac{1}{n^{\frac{k}{2}}}\\\label{eqn2.12}
	 & \le \left(1-\frac{1}{\sqrt{2}}\right)^{-1}\sqrt{\frac{\pi}{2}} \cosh(4 \pi) \cdot n^{-\frac{N+1}{2}} \le   2 \sqrt{2 \pi}\cosh(4 \pi)\cdot n^{-\frac{N+1}{2}}.
\end{align}
\endgroup 
Combining \eqref{eqn2.9} and \eqref{eqn2.12}, we conclude the proof.
\end{proof}
\begin{lemma}\label{Glem2}
	For $N \geq 1$, $\a \in {\mathbb{R}}_{\geq 1}$ and $n \geq \left\lceil 4\a \right\rceil$, we have 
	\begin{equation*}
		\left(1-\frac{1}{n}\right)^{-\a}= \sum_{m=0}^{N} \frac{A_m(\a)}{n^{\frac{m}{2}}}+ O_{\leq E_3(\a, N)} \left(n^{-\frac{N+1}{2}}\right)\ \ \text{with}\ \ E_3\left(\a,N\right):=2\cdot \a^{\frac{N+1}{2}},
	\end{equation*}
and $\left(A_m(\a)\right)_{m\ge 0}$ is given by \eqref{eqn2.13}.
\end{lemma}
	\begin{proof}
		From the Taylor expansion of $\left(1-\frac{1}{n}\right)^{-\alpha}$, we get
		\begin{equation}\label{Glem2eqn1}
			\left(1-\frac{1}{n}\right)^{-\a}= \sum_{k=0}^{\infty} \binom{-\a}{k} (-1)^k \left(\frac{1}{\sqrt{n}}\right)^{2k}=: \sum_{m=0}^{\infty} \frac{A_m(\a)}{n^{\frac{m}{2}}}=\sum_{m=0}^{N} \frac{A_m(\a)}{n^{\frac{m}{2}}}+\sum_{m=N+1}^{\infty} \frac{A_m(\a)}{n^{\frac{m}{2}}},
		\end{equation}
	where 
	\begin{equation}\label{eqn2.13}
		A_m(\a)=\begin{cases}
			(-1)^{\frac{m}{2}}\binom{-\a}{\frac{m}{2}}, &\quad \text{if}\  m \equiv 0 \left(\text{mod}\ 2\right),\\
			0, &\quad\ \text{otherwise}. 
		\end{cases}
	\end{equation}
So for $m=2\ell$, we have
\begin{equation}\label{eqn2.14}
\binom{-\a}{\ell} (-1)^\ell=\prod_{j=1}^{\ell}\frac{\a+j-1}{j}\le \a^{\ell}\ \ \left(\text{as}\ \a\ge 1\right),
\end{equation}
and thus from \eqref{eqn2.13} and \eqref{eqn2.14}, it directly implies that for $m\in \mathbb{N}$, 
\begin{equation}\label{Glem2eqn2}
\left|A_m(\a)\right|\le \a^{\frac{m}{2}}.
\end{equation}
Using \eqref{Glem2eqn2} and $n \ge \left\lceil 4\a \right\rceil$, it follows that
\begin{equation}\label{eqn2.15}
\left|\sum_{m=N+1}^{\infty}\frac{A_m(\a)}{n^{\frac{m}{2}}}\right|\le\sum_{m=N+1}^{\infty} \left(\frac{\a}{n}\right)^{\frac{m}{2}}\le E_3\left(\a,N\right)\cdot n^{-\frac{N+1}{2}}, 
\end{equation}
and applying \eqref{eqn2.15} to \eqref{Glem2eqn1}, we conclude the proof.
\end{proof}
Next, we combine Lemmas \ref{Glem1} and \ref{Glem2} $\left(\text{with}\ \a\mapsto \frac{27}{4}\right)$ to estimate an error bound for the expansion of the factor $\frac{e^{4 \pi (\sqrt{n-1}-\sqrt{n})}}{(1-\frac{1}{n})^{\frac{27}{4}}}$. 
\begin{lemma}\label{Glem3}
	For $N \ge 1$ and $n \ge 27$, we have
	\begin{equation*}
	\frac{e^{4 \pi (\sqrt{n-1}-\sqrt{n})}}{(1-\frac{1}{n})^{\frac{27}{4}}}=\sum_{m=0}^{N}\frac{\widehat{B}_m}{n^{\frac{m}{2}}}+O_{\le E_4(N)}\left(n^{-\frac{N+1}{2}}\right),
	\end{equation*}
where for $m\in \mathbb{N}_0$,
\begin{equation}\label{Glem3def1}
\widehat{B}_{2m}:=\sum_{\ell=0}^{m}A_1(\ell)A_{2m-2\ell}\ \ \text{and}\ \ \widehat{B}_{2m+1}:=\sum_{\ell=0}^{m}A_2(\ell)A_{2m-2\ell}\ \ \text{with}\ A_m:=A_m\left(\frac{27}{4}\right),
\end{equation}
and 
\begin{equation*}
E_4(N):=3\cdot E_2\cdot E_3\left(\frac{27}{4}, N\right)\cdot N.
\end{equation*}
\end{lemma}
\begin{proof}
Applying Lemmas \ref{Glem1} and \ref{Glem2} with $\a\mapsto\frac{27}{4}$, and following \eqref{Glem3def1}, we get
\begingroup
\allowdisplaybreaks
\begin{align}\nonumber
\frac{e^{4 \pi (\sqrt{n-1}-\sqrt{n})}}{(1-\frac{1}{n})^{\frac{27}{4}}}&=\left(\sum_{k=0}^{N}\frac{T_k}{n^{\frac{k}{2}}} + O_{\leq E_2}\left(n^{-\frac{N+1}{2}}\right)\right)\left(\sum_{k=0}^{N} \frac{A_k(\frac{27}{4})}{n^{\frac{k}{2}}}+ O_{\leq E_3\left(\frac{27}{4}, N\right)} \left(n^{-\frac{N+1}{2}}\right)\right)\\\nonumber
&=\sum_{k=0}^{N}\frac{T_k}{n^{\frac{k}{2}}}\sum_{k=0}^{N}\frac{A_k}{n^{\frac{k}{2}}}+\sum_{k=0}^{N}\frac{T_k}{n^{\frac{k}{2}}}\cdot O_{\leq E_3\left(\frac{27}{4}, N\right)} \left(n^{-\frac{N+1}{2}}\right)+\sum_{k=0}^{N}\frac{A_k}{n^{\frac{k}{2}}}\cdot O_{\leq E_2} \left(n^{-\frac{N+1}{2}}\right)\\\label{Glem3eqn1}
&\hspace{8 cm}+O_{\le E_2\cdot E_3\left(\frac{27}{4}, N\right)}\left(n^{-N-1}\right).
\end{align}	
\endgroup	
Next, applying \eqref{Glem1def1} and \eqref{Glem1def2}, it follows that
\begingroup
\allowdisplaybreaks
\begin{align}\nonumber
&\sum_{k=0}^{N}\frac{T_k}{n^{\frac{k}{2}}}\sum_{k=0}^{N}\frac{A_k}{n^{\frac{k}{2}}}=\left(\sum_{k=0}^{\left\lfloor \frac{N}{2}\right\rfloor}\frac{T_{2k}}{n^{k}}+\sum_{k=0}^{\left\lfloor \frac{N-1}{2}\right\rfloor}\frac{T_{2k+1}}{n^{k+\frac{1}{2}}}\right)\sum_{k=0}^{\left\lfloor\frac{N}{2}\right\rfloor}\frac{A_{2k}}{n^{k}}\\\nonumber
&=
\sum_{k=0}^{\left\lfloor \frac{N}{2}\right\rfloor}\frac{A_1(k)}{n^{k}}\sum_{k=0}^{\left\lfloor \frac{N}{2}\right\rfloor}\frac{A_{2k}}{n^{k}}+\sum_{k=0}^{\left\lfloor \frac{N-1}{2}\right\rfloor}\frac{A_2(k)}{n^{k+\frac{1}{2}}}\sum_{k=0}^{\left\lfloor\frac{N}{2}\right\rfloor}\frac{A_{2k}}{n^{k}}\\\nonumber
&=\sum_{k=0}^{N}\frac{\widehat{B}_k}{n^{\frac k2}}+\chi_{N} \frac{A_{2\left\lfloor\frac{N}{2}\right\rfloor}}{n^{\left\lfloor\frac{N}{2}\right\rfloor+\frac 12}}\sum_{k=0}^{\left\lfloor\frac{N-1}{2}\right\rfloor}\frac{A_2(k)}{n^{k}}+n^{-\left\lfloor\frac{N}{2}\right\rfloor-1} \sum_{k=0}^{\left\lfloor\frac{N}{2}\right\rfloor-1}\sum_{m=k}^{\left\lfloor\frac{N}{2}\right\rfloor-1}\frac{A_1(m+1)A_{2\left\lfloor\frac{N}{2}\right\rfloor+2k-2m}}{n^k}\\\label{Glem3eqn2}
&\hspace{4 cm}+ n^{-\left\lfloor\frac{N-1}{2}\right\rfloor-\frac 32} \sum_{k=0}^{\left\lfloor\frac{N-1}{2}\right\rfloor-1}\sum_{m=k}^{\left\lfloor\frac{N-1}{2}\right\rfloor-1}\frac{A_2(m+1)A_{2\left\lfloor\frac{N-1}{2}\right\rfloor+2k-2m}}{n^k},
\end{align}	
\endgroup	
where $\chi_{N}:=1$ if $N$ is even and $0$ otherwise. Now we estimate each error factor involving terms which are of the form $n^{-\frac{N+1}{2}}$ in \eqref{Glem3eqn2}. First, applying \eqref{eqn2.11} and \eqref{Glem2eqn2} with $\a\mapsto\frac{27}{4}$, and by \eqref{Glem1def2} with $k\mapsto 0$, we get
\begin{equation}\label{Glem3eqn3}
\left|\chi_{N} \frac{A_{2\left\lfloor\frac{N}{2}\right\rfloor}}{n^{\left\lfloor\frac{N}{2}\right\rfloor+\frac 12}}\sum_{k=0}^{\left\lfloor\frac{N-1}{2}\right\rfloor}\frac{A_2(k)}{n^{k}}\right|\le \frac{E_3\left(\frac{27}{4}, N\right)\left(2\pi+\sqrt{2}\cdot E_2\right)}{3^{\frac 32}}n^{-\frac{N+1}{2}}.
\end{equation}
Next, applying \eqref{eqn2.10} and \eqref{Glem2eqn2} with $\a\mapsto\frac{27}{4}$, it follows that
\begingroup
\allowdisplaybreaks
\begin{align}\nonumber
&\left|n^{-\left\lfloor\frac{N}{2}\right\rfloor-1} \sum_{k=0}^{\left\lfloor\frac{N}{2}\right\rfloor-1}\sum_{m=k}^{\left\lfloor\frac{N}{2}\right\rfloor-1}\frac{A_1(m+1)A_{2\left\lfloor\frac{N}{2}\right\rfloor+2k-2m}}{n^k}\right|\\\nonumber
&\hspace{4.5 cm} \le 2\sqrt{\pi}\sinh\left(4\pi\right) \sum_{k=0}^{\left\lfloor\frac{N}{2}\right\rfloor-1} \sum_{m=k}^{\left\lfloor\frac{N}{2}\right\rfloor-1} \frac{1}{(m+1)^{\frac{3}{2}}}\left(\frac{27}{4}\right)^{\left\lfloor \frac{N}{2}\right\rfloor+k-m}n^{-\frac{N+1}{2}}\\\label{Glem3eqn4}
&\hspace{4.5 cm} \le \frac{\left(\frac 32\right)^{\frac 32}\cdot E_2\cdot E_3\left(\frac{27}{4},N\right)\cdot N}{23}n^{-\frac{N+1}{2}}.
\end{align}
\endgroup 
Similarly, using \eqref{eqn2.11} and \eqref{Glem2eqn2} with $\a\mapsto\frac{27}{4}$, we have
\begin{equation}\label{Glem3eqn5}
\left|n^{-\left\lfloor\frac{N-1}{2}\right\rfloor-\frac 32} \sum_{k=0}^{\left\lfloor\frac{N-1}{2}\right\rfloor-1}\sum_{m=k}^{\left\lfloor\frac{N-1}{2}\right\rfloor-1}\frac{A_2(m+1)A_{2\left\lfloor\frac{N-1}{2}\right\rfloor+2k-2m}}{n^k}\right|\le \frac{\left(\frac 32\right)^{\frac 32}\cdot E_2\cdot E_3\left(\frac{27}{4}, N\right)\cdot N}{23}n^{-\frac{N+1}{2}}.
\end{equation}
Applying \eqref{Glem3eqn3}-\eqref{Glem3eqn5} to \eqref{Glem3eqn2}, for all $N\ge 1$ and $n\ge 2$, we obtain
\begin{equation}\label{Glem3eqn6}
\sum_{k=0}^{N}\frac{T_k}{n^{\frac{k}{2}}}\sum_{k=0}^{N}\frac{A_k}{n^{\frac{k}{2}}}=\sum_{k=0}^{N}\frac{\widehat{B}_k}{n^{\frac k2}}+O_{\le E_{4,1}(N)}\left(n^{-\frac{N+1}{2}}\right)\ \text{with}\ E_{4,1}(N)=\frac{E_2\cdot E_3\left(\frac{27}{4},N\right)\cdot N}{2}.
\end{equation}
To finish the proof, it remains to estimate the three terms involving the factor $n^{-\frac{N+1}{2}}$ in \eqref{Glem3eqn1}. It is clear that for all $n\ge 2$,
\begin{equation}\label{Glem3eqn7}
O_{\le E_2\cdot E_3\left(\frac{27}{4}, N\right)}\left(n^{-N-1}\right)=O_{\le E_2\cdot E_3\left(\frac{27}{4}, N\right)}\left(n^{-\frac{N+1}{2}}\right).
\end{equation}
Moving forward, applying \eqref{Glem2eqn2} with $\a\mapsto\frac{27}{4}$, for all $n\ge 27$, it follows that
\begin{equation*}\label{Glem3eqn8}
E_2\cdot n^{-\frac{N+1}{2}}\left|\sum_{k=0}^{N}\frac{A_k}{n^{\frac{k}{2}}}\right|\le E_2\cdot n^{-\frac{N+1}{2}}\sum_{k\ge 0}\left(\frac{27}{4n}\right)^{\frac k2}\le 2\ E_2\cdot n^{-\frac{N+1}{2}},
\end{equation*}
which implies
\begin{equation}
\sum_{k=0}^{N}\frac{A_k}{n^{\frac{k}{2}}}\cdot O_{\leq E_2} \left(n^{-\frac{N+1}{2}}\right)=O_{\le 2\ E_2}\left(n^{-\frac{N+1}{2}}\right).
\end{equation}
Analogously, following \eqref{Akfinalbound}, for all $n\ge 27$, we obtain
\begin{equation}\label{Glem3eqn9}
\sum_{k=0}^{N}\frac{T_k}{n^{\frac{k}{2}}}\cdot O_{\leq E_3\left(\frac{27}{4}, N\right)} \left(n^{-\frac{N+1}{2}}\right)=O_{\le E_2 \cdot E_3\left(\frac{27}{4}, N\right)}\left(n^{-\frac{N+1}{2}}\right).
\end{equation}
Finally combining \eqref{Glem3eqn7}-\eqref{Glem3eqn9} to \eqref{Glem3eqn6} and then plugging into \eqref{Glem3eqn1}, we conclude the proof.
\end{proof}
For $m\in \mathbb{N}_0$, define
\begin{align}\nonumber
C_{2m}&:=(-1)^m\sum_{k=0}^{m}\binom{-k}{m-k}\frac{(-1)^{k}a_{2k}(13)}{(4\pi)^{2k}},\\\label{Glem4def1} C_{2m+1}&:=(-1)^{m+1}\sum_{k=0}^{m}\binom{-\frac{2k+1}{2}}{m-k}\frac{(-1)^{k}a_{2k+1}(13)}{(4\pi)^{2k+1}}.
\end{align}
\begin{lemma}\label{Glem4}
	For $N \geq1$ and $n \geq 138$, we have 
	\begin{equation*}
	\sum_{m=0}^{N} \frac{(-1)^m a_m(13)}{(4\pi)^m(n-1)^{\frac m2}} =\sum_{m=0}^{N} \frac{C_m}{n^{\frac{m}{2}}}+O_{\leq E_5(N)}\left(n^{-\frac{N+1}{2}}\right)\ \ \text{with}\ \ E_5(N):=17\cdot 69^{\frac{N+1}{2}}\cdot N!.
	\end{equation*}
\end{lemma}
\begin{proof}
Setting
$$b_k:=\frac{(-1)^k a_k(13)}{(4\pi)^k}, c_{k,\ell}:=(-1)^\ell\binom{-\frac{k}{2}}{\ell},\ \text{and}\ z:=n^{-\frac 12},$$
we rewrite the sum as
\begingroup
\allowdisplaybreaks
\begin{align}\nonumber
&\sum_{k=0}^{N} \frac{(-1)^k a_k(13)}{(4\pi)^k(n-1)^{\frac k2}}=\sum_{k=0}^{N} \frac{(-1)^k a_k(13)}{(4\pi)^k n^{\frac{k}{2}}} \left(1-\frac{1}{n}\right)^{-\frac{k}{2}}=\sum_{k=0}^{N} \sum_{\ell=0}^{\infty} b_k c_{k,\ell} z^{2\ell+k}\\\nonumber
&= \sum_{k=0}^{\left\lfloor \frac{N}{2} \right\rfloor} \sum_{\ell=0}^{\infty} b_{2k} c_{2k,\ell} z^{2\ell+2k} +\sum_{k=0}^{\left\lfloor \frac{N-1}{2} \right\rfloor} \sum_{\ell=0}^{\infty} b_{2k+1} c_{2k+1,\ell} z^{2\ell+2k+1}\\\nonumber
&=\sum_{m=0}^{\left\lfloor\frac{N}{2}\right\rfloor}\sum_{k=0}^{m} b_{2k} c_{2k,m-k} z^{2m}+\sum_{m=0}^{\left\lfloor \frac{N-1}{2} \right\rfloor } \sum_{k=0}^{m} b_{2k+1} c_{2k+1,m-k} z^{2m+1}+\sum_{m=\left\lfloor \frac{N}{2} \right\rfloor+1}^{\infty}\sum_{k=0}^{\left\lfloor \frac{N}{2} \right\rfloor} b_{2k} c_{2k,m-k} z^{2m}\\\nonumber
&\hspace{7 cm}+\sum_{m=\left\lfloor \frac{N-1}{2} \right\rfloor+1}^{\infty} \sum_{k=0}^{\left\lfloor \frac{N-1}{2} \right\rfloor} b_{2k+1} c_{2k+1,m-k} z^{2m+1}\\\nonumber
&=\sum_{m=0}^{N}C_mz^m+\sum_{m=\left\lfloor \frac{N}{2} \right\rfloor+1}^{\infty}\sum_{k=0}^{\left\lfloor \frac{N}{2} \right\rfloor} b_{2k} c_{2k,m-k} z^{2m}+\sum_{m=\left\lfloor \frac{N-1}{2} \right\rfloor+1}^{\infty} \sum_{k=0}^{\left\lfloor \frac{N-1}{2} \right\rfloor} b_{2k+1} c_{2k+1,m-k} z^{2m+1}\\\label{Glem4eqn1}
&=:\sum_{m=0}^{N} C_m z^m +E_1\left(N,z\right)+ E_2\left(N,z\right).
\end{align}
\endgroup 
Next we estimate bounds for $E_1(N,z)$ and $E_2(N,z)$. First we observe that for $k\in \mathbb{N}_0$,
\begin{equation}\label{akestim}
\left|a_k(13)\right|= \frac{\left|\binom{\frac{25}{2}}{k}\right|\left(\frac{27}{2}\right)_k}{2^k}\le \left(\frac{27}{2\sqrt{2}}\right)^{2k}k!. 
\end{equation}
Now, we bound $\left|E_1(N,z)\right|$ as follows
\begingroup
\allowdisplaybreaks
\begin{align}\nonumber
\left|E_1(N,z)\right|&= \left|\sum_{m=\lfloor \frac{N}{2} \rfloor+1}^{\infty} \sum_{k=0}^{\left\lfloor \frac{N}{2} \right\rfloor} b_{2k} c_{2k,m-k} z^{2m}\right|\le \sum_{m=\left\lfloor \frac{N}{2} \right\rfloor+1}^{\infty} \sum_{k=0}^{\left\lfloor \frac{N}{2} \right\rfloor} |b_{2k}| |c_{2k,m-k}| z^{2m}\ \ \left(\text{by}\ \eqref{Glem4eqn1}\right)\\\nonumber
&=\sum_{m=\left\lfloor \frac{N}{2} \right\rfloor+1}^{\infty}\frac{1}{n^m} \sum_{k=0}^{\left\lfloor \frac{N}{2} \right\rfloor}\left|\frac{a_{2k}(13)}{(4\pi)^{2k}}\right|\left|\binom{-k}{m-k}(-1)^{m-k}\right|\\\nonumber
&=\sum_{m=\left\lfloor \frac{N}{2} \right\rfloor+1}^{\infty}\frac{1}{n^m} \sum_{k=1}^{\left\lfloor \frac{N}{2} \right\rfloor}\binom{m-1}{k-1}\left|\frac{a_{2k}(13)}{(4\pi)^{2k}}\right|\\\nonumber
&\le \sum_{m=\left\lfloor \frac{N}{2} \right\rfloor+1}^{\infty}\frac{1}{n^m} \sum_{k=1}^{\left\lfloor \frac{N}{2} \right\rfloor}\binom{m-1}{k-1} \left(\left(\frac{27}{2\sqrt{2}}\right)^4\cdot \frac{1}{16\pi^2}\right)^k(2k)!\ \ \left(\text{by}\ \eqref{akestim}\right)\\\nonumber
&\le N!\sum_{m=\left\lfloor \frac{N}{2} \right\rfloor+1}^{\infty}\frac{1}{n^m} \sum_{k=1}^{\left\lfloor \frac{N}{2} \right\rfloor}\binom{m-1}{k-1}\cdot 53^k\le N!\sum_{m=\left\lfloor \frac{N}{2} \right\rfloor+1}^{\infty}\frac{1}{n^m} \sum_{k=1}^{m}\binom{m-1}{k-1}\cdot 53^k\\\nonumber
&=\frac{53}{54}\cdot N!\sum_{m=\left\lfloor \frac{N}{2} \right\rfloor+1}^{\infty}\left(\frac{54}{n}\right)^m\le N! \frac{54^{\left\lfloor\frac{N}{2}\right\rfloor+1}}{n^{\frac{N+1}{2}}}\sum_{m\ge 0}\left(\frac{54}{n}\right)^m\\\label{Glem4eqn2}
&\le N!\frac{54^{\left\lfloor\frac{N}{2}\right\rfloor+1}}{n^{\frac{N+1}{2}}}\sum_{m\ge 0}\left(\frac{1}{2}\right)^m\ \ \left(\text{for all}\ n\ge 108\right)= 2\cdot N! \cdot 54^{\left\lfloor\frac{N}{2}\right\rfloor+1}n^{-\frac{N+1}{2}}.
\end{align}
\endgroup
In addition to that, for $\left|E_2(N,z)\right|$, we get
\begingroup
\allowdisplaybreaks
\begin{align}\nonumber
\left|E_2(N,z)\right|&= \left|\sum_{m=\left\lfloor \frac{N-1}{2} \right\rfloor+1}^{\infty} \sum_{k=0}^{\left\lfloor \frac{N-1}{2} \right\rfloor} b_{2k+1} c_{2k+1,m-k} z^{2m}\right|\\\nonumber
&\le \sum_{m=\left\lfloor \frac{N-1}{2} \right\rfloor+1}^{\infty} \frac{1}{n^{m+\frac{1}{2}}} \sum_{k=0}^{\left\lfloor \frac{N-1}{2} \right\rfloor} |b_{2k+1}| |c_{2k+1,m-k}|\ \ \left(\text{by}\ \eqref{Glem4eqn1}\right)\\\nonumber
&=\left(\frac{27}{2 \sqrt{2}}\right)^2 \frac{1}{4\pi} \sum_{m=\left\lfloor \frac{N-1}{2} \right\rfloor+1}^{\infty}\frac{1}{n^{m+\frac{1}{2}}} \sum_{k=0}^{\left\lfloor \frac{N-1}{2} \right\rfloor}\left|\frac{a_{2k+1}(13)}{(4\pi)^{2k+1}}\right|\left|\frac{\binom{2m}{2k}\binom{2m-2k}{m-k}}{4^{m-k}\binom{m}{k}}\right|\\\nonumber
&\le \frac{3^6}{2^5\pi} \sum_{m=\left\lfloor \frac{N-1}{2} \right\rfloor+1}^{\infty}\frac{1}{n^{m+\frac{1}{2}}} \sum_{k=0}^{\left\lfloor \frac{N-1}{2} \right\rfloor}\left|\frac{\binom{2m}{2k}\binom{2m-2k}{m-k}}{4^{m-k}\binom{m}{k}}\right|\left(\frac{3^{12}}{2^{10}\pi^2}\right)^k (2k+1)!\ \ \left(\text{by}\ \eqref{akestim}\right)\\\nonumber
& \le \frac{3^6}{2^5\pi}N!\sum_{m=\left\lfloor \frac{N-1}{2} \right\rfloor+1}^{\infty}\frac{1}{n^{m+\frac{1}{2}}} \sum_{k=0}^{m} \left|\frac{\binom{2m}{2k}\binom{2m-2k}{m-k}}{4^{m-k}\binom{m}{k}}\right|\left(\frac{3^{12}}{2^{10}\pi^2}\right)^k \\\nonumber
&\le 8\cdot N! \sum_{m=\left\lfloor \frac{N-1}{2} \right\rfloor+1}^{\infty}\frac{1}{n^{m+\frac{1}{2}}} \sum_{k=0}^{m} \binom{2m}{2k} 53^k\le 8\cdot N!\sum_{m=\lfloor \frac{N-1}{2} \rfloor+1}^{\infty}\frac{1}{n^{m+\frac{1}{2}}} (\sqrt{53}+1)^{2m}\\\nonumber
&\le 8\cdot N!\sum_{m=\lfloor \frac{N-1}{2} \rfloor+1}^{\infty}\frac{69^m}{n^{m+\frac{1}{2}}}\le 8\cdot N! \frac{69^{\left\lfloor \frac{N-1}{2} \right\rfloor+1}}{n^{\frac{N+1}{2}}} \sum_{m=0}^{\infty}\frac{1}{2^m}\ \ \left(\text{for all}\ n\ge 138\right)\\\label{Glem4eqn3} 
&= 16\cdot N! \cdot 69^{\left\lfloor \frac{N-1}{2} \right\rfloor+1}n^{-\frac{N+1}{2}}.
\end{align}
\endgroup
Applying \eqref{Glem4eqn2} and \eqref{Glem4eqn3} to \eqref{Glem4eqn1}, we conclude the proof.
\end{proof}
\begin{lemma}\label{Glem5}
	 For $N \ge 1$ and $n \ge 138$, we have
	\begin{equation*}
		\left(\sum_{m=0}^{N} \frac{(-1)^k a_m(13)}{(4\pi)^m(n-1)^{\frac m2}}+O_{\le E_{N,1}}\left(n^{-\frac{N+1}{2}}\right) \right) \left(1+O_{\le 1}\left(n^{-\frac{N+1}{2}}\right)\right)=\sum_{m=0}^{N} \frac{C_m}{n^{\frac{m}{2}}}+O_{\leq E_6(N)}\left(n^{-\frac{N+1}{2}}\right),
	\end{equation*}
	with 
	\begin{equation*}
	E_6(N):=2\cdot\left(E_5(N)+E_{N,1}\right)+18\cdot N!.
	\end{equation*}
\end{lemma}
\begin{proof}
Applying Lemma \ref{Glem4}, for all $n\ge 138$, we get
\begin{equation*}
\sum_{m=0}^{N} \frac{(-1)^k a_m(13)}{(4\pi)^m(n-1)^{\frac m2}}+O_{\le E_{N,1}}\left(n^{-\frac{N+1}{2}}\right)=\sum_{m=0}^{N} \frac{C_m}{n^{\frac{m}{2}}}+O_{\leq {E_5(N)+E_{N,1}}}\left(n^{-\frac{N+1}{2}}\right),
\end{equation*}
and hence,
\begin{align}\nonumber
&\left(\sum_{m=0}^{N} \frac{(-1)^k a_m(13)}{(4\pi)^m(n-1)^{\frac m2}}+O_{\le E_{N,1}}\left(n^{-\frac{N+1}{2}}\right) \right) \left(1+O_{\le 1}\left(n^{-\frac{N+1}{2}}\right)\right)\\\nonumber
&=\sum_{m=0}^{N} \frac{C_m}{n^{\frac{m}{2}}}+\sum_{m=0}^{N} \frac{C_m}{n^{\frac{m}{2}}}\cdot O_{\le 1}\left(n^{-\frac{N+1}{2}}\right)+O_{\leq {E_5(N)+E_{N,1}}}\left(n^{-\frac{N+1}{2}}\right)+O_{\leq {E_5(N)+E_{N,1}}}\left(n^{-N-1}\right)\\\label{Glem5eqn1}
&=\sum_{m=0}^{N} \frac{C_m}{n^{\frac{m}{2}}}+\sum_{m=0}^{N} \frac{C_m}{n^{\frac{m}{2}}}\cdot O_{\le 1}\left(n^{-\frac{N+1}{2}}\right)+O_{\leq {2\left(E_5(N)+E_{N,1}\right)}}\left(n^{-\frac{N+1}{2}}\right).
\end{align}
Furthermore, for all $n\ge 138$, we obtain
\begingroup
\allowdisplaybreaks
\begin{align*}
&\left|\sum_{m=0}^{N} \frac{C_m}{n^{\frac{m}{2}}}\right|\le \sum_{m=0}^{\left\lfloor \frac{N}{2} \right\rfloor} \frac{\left |C_{2m}\right|} {n^{m}}+\sum_{m=0}^{\left\lfloor \frac{N-1}{2} \right\rfloor} \frac{\left|C_{2m+1}\right|} {n^{m+\frac{1}{2}}}\\
&\le \sum_{m=0}^{\left\lfloor \frac{N}{2} \right\rfloor}\frac{1}{n^m}\sum_{k=0}^{m}\binom{m-1}{k-1}\frac{\left|a_{2k}(13)\right|}{(4\pi)^{2k}}+\sum_{m=0}^{\left\lfloor \frac{N-1}{2} \right\rfloor} \frac{1}{n^{m+\frac 12}}\sum_{k=0}^{m}\left|\binom{-\frac{2k+1}{2}}{m-k}(-1)^{m-k}\right|\frac{\left|a_{2k+1}(13)\right|}{(4\pi)^{2k+1}}\\
&\hspace{12 cm} \left(\text{by Lemma}\ \ref{Glem4}\right)\\
&=\sum_{m=0}^{\left\lfloor \frac{N}{2} \right\rfloor}\frac{1}{n^m}\sum_{k=0}^{m}\binom{m-1}{k-1}\frac{\left|a_{2k}(13)\right|}{(4\pi)^{2k}}+\sum_{m=0}^{\left\lfloor \frac{N-1}{2} \right\rfloor} \frac{1}{n^{m+\frac 12}}\sum_{k=0}^{m}\frac{\binom{2m}{2k}\binom{2m-2k}{m-k}}{4^{m-k}\binom{m}{k}}\frac{\left|a_{2k+1}(13)\right|}{(4\pi)^{2k+1}}\\
&\le \sum_{m=0}^{\left\lfloor \frac{N}{2} \right\rfloor}\frac{1}{n^m}\sum_{k=0}^{m}\binom{m-1}{k-1}\left(\frac{3^{12}}{2^{10}\pi^2}\right)^k (2k)!\\
&\hspace{3 cm}+\frac{3^6}{2^5\pi}\sum_{m=0}^{\left\lfloor \frac{N-1}{2} \right\rfloor} \frac{1}{n^{m+\frac 12}}\sum_{k=0}^{m}\frac{\binom{2m}{2k}\binom{2m-2k}{m-k}}{4^{m-k}\binom{m}{k}}\left(\frac{3^{12}}{2^{10}\pi^2}\right)^k (2k+1)!\ \ \left(\text{by}\ \eqref{akestim}\right)\\
&\le N!\sum_{m=0}^{\left\lfloor \frac{N}{2} \right\rfloor}\frac{1}{n^m}\sum_{k=0}^{m}\binom{m-1}{k-1}53^k+8\cdot N!\sum_{m=0}^{\left\lfloor \frac{N-1}{2} \right\rfloor} \frac{1}{n^{m+\frac 12}}\sum_{k=0}^{m}\binom{2m}{2k}53^k\ \ \left(\text{by}\ \eqref{binomestim}\right)\\
&\le N!\sum_{m=0}^{\left\lfloor \frac{N}{2} \right\rfloor}\left(\frac{54}{n}\right)^m+8\cdot N!\sum_{m=0}^{\left\lfloor \frac{N-1}{2} \right\rfloor}\left(\frac{69}{n}\right)^m\le N!\left(\sum_{m\ge 0}\left(\frac{54}{n}\right)^m+8\sum_{m\ge 0}\left(\frac{69}{n}\right)^m\right)\\
&\le N!\left(\sum_{m\ge 0}\left(\frac{2}{5}\right)^m+8\sum_{m\ge 0}\left(\frac{1}{2}\right)^m\right)\ \ \left(\text{for}\ n\ge 138\right)\le 18\cdot N!,
\end{align*}
\endgroup 
and so
\begin{equation}\label{Glem5eqn2}
\sum_{m=0}^{N} \frac{C_m}{n^{\frac{m}{2}}}\cdot O_{\le 1}\left(n^{-\frac{N+1}{2}}\right)=O_{\le 18\cdot N!}\left(n^{-\frac{N+1}{2}}\right).
\end{equation}
Applying \eqref{Glem5eqn2} to \eqref{Glem5eqn1}, we conclude the proof.
\end{proof}
With help of Lemmas \ref{Glem3} and \ref{Glem5}, the asymptotic expansion of $p_{24}(n)$ in \eqref{eqn2.4} takes the following shape: for all $n\ge n(N)$,
\begin{equation}\label{asympexp1}
p_{24}(n)=\frac{e^{4 \pi \sqrt{n}}}{\sqrt{2}\cdot n^{\frac{27}{4}}}\left(\sum_{m=0}^{N}\frac{\widehat{B}_m}{n^{\frac{m}{2}}}+O_{\le E_4(N)}\left(n^{-\frac{N+1}{2}}\right)\right)\left(\sum_{m=0}^{N} \frac{C_m}{n^{\frac{m}{2}}}+O_{\leq E_6(N)}\left(n^{-\frac{N+1}{2}}\right)\right).
\end{equation} 
Now we are ready to prove Theorem \ref{thm1.1}.

\emph{Proof of Theorem \ref{thm1.1}:}
To begin with, for all $n\ge n(N)$, we first expand
\begingroup
\allowdisplaybreaks
\begin{align}\nonumber
&\left(\sum_{m=0}^{N}\frac{\widehat{B}_m}{n^{\frac{m}{2}}}+O_{\le E_4(N)}\left(n^{-\frac{N+1}{2}}\right)\right)\left(\sum_{m=0}^{N} \frac{C_m}{n^{\frac{m}{2}}}+O_{\leq E_6(N)}\left(n^{-\frac{N+1}{2}}\right)\right)\\\nonumber
&=\sum_{m=0}^{N}\sum_{k=0}^{m}\frac{\widehat{B}_k\cdot C_{m-k}}{n^{\frac m2}}+n^{-\frac{N+1}{2}}\sum_{m=0}^{N-1}\frac{1}{n^{\frac m2}}\sum_{k=m}^{N-1}\widehat{B}_{k+1}\cdot C_{N+m-k}+\sum_{m=0}^{N}\frac{\widehat{B}_m}{n^{\frac{m}{2}}}\cdot O_{\leq E_6(N)}\left(n^{-\frac{N+1}{2}}\right)\\\nonumber
&\hspace{5.5 cm}+\sum_{m=0}^{N} \frac{C_m}{n^{\frac{m}{2}}}\cdot O_{\le E_4(N)}\left(n^{-\frac{N+1}{2}}\right)+O_{\le {E_4(N)\cdot E_6(N)}}\left(n^{-\frac{N+1}{2}}\right)\\\label{thmeqn1}
&=:\sum_{m=0}^{N}\frac{\widetilde{B}_m}{n^{\frac m2}}+\widetilde{E}_1(n, N)+\widetilde{E}_2(n, N)+\widetilde{E}_3(n, N)+O_{\le {E_4(N)\cdot E_6(N)}}\left(n^{-\frac{N+1}{2}}\right),
\end{align}
\endgroup
where for $m\in \mathbb{N}_0$,
\begin{equation}\label{finalcoeffdef}
\widetilde{B}_m=\sum_{k=0}^{m}\widehat{B}_k\cdot C_{m-k}.
\end{equation}
Now it remains to compute the error bounds for the sums $\widetilde{E}_k(n, N)$ for $1\le k \le 3$. We start by providing estimates for $\widehat{B}_m$ and $C_m$ for all $m\in \mathbb{N}_0$. For $m\in \mathbb{N}$, we see that
\begingroup
\allowdisplaybreaks
\begin{align}\nonumber
\left|\widehat{B}_{2m}\right|&=\left|\sum_{\ell=0}^{m}A_1(\ell)A_{2m-2\ell}\right|\ \ \left(\text{by}\ \eqref{Glem3def1}\right)\le A_1(0)\left|A_{2m}\right|+\sum_{\ell=1}^{m-1}\left|A_1(\ell)\right|\cdot\left|A_{2m-2\ell}\right|+\left|A_1(m)\right|A_0\\\nonumber
&=\left|A_{2m}\right|+\sum_{\ell=1}^{m-1}\left|A_1(\ell)\right|\cdot\left|A_{2m-2\ell}\right|+\left|A_1(m)\right|\\\nonumber
&\hspace{5 cm} \left(\text{by}\ \eqref{Glem1def1}, \eqref{Glem3def1}, \text{and}\ \eqref{eqn2.13}\ \text{with}\ \left(\a,m\right)\mapsto \left(\frac{27}{4},0\right)\right)\\\nonumber
&\le \left(\frac{27}{4}\right)^m+2\sqrt{\pi}\sinh(4\pi)\sum_{\ell=1}^{m-1}\frac{\left(\frac{27}{4}\right)^{m-\ell}}{\ell^{\frac 32}}+\frac{2\sqrt{\pi}\sinh(4\pi)}{m^{\frac 32}}\\\nonumber
&\hspace{8 cm} \left(\text{by}\ \eqref{eqn2.10}\ \text{and}\ \eqref{Glem2eqn2}\ \text{with}\ \a\mapsto\frac{27}{4}\right)\\\nonumber
&=\left(\frac{27}{4}\right)^m\left(1+2\sqrt{\pi}\sinh(4\pi)\sum_{\ell=1}^{m}\frac{\left(\frac{4}{27}\right)^{\ell}}{\ell^{\frac 32}}\right)\le \left(\frac{27}{4}\right)^m\left(1+2\sqrt{\pi}\sinh(4\pi)\sum_{\ell\ge 1}\left(\frac{4}{27}\right)^{\ell}\right)\\\label{Bhatevenestim}
&=\left(\frac{27}{4}\right)^m\left(1+\frac{8\sqrt{\pi}\sinh(4\pi)}{23}\right).
\end{align}
\endgroup 
Similarly, for $m\in \mathbb{N}$, using \eqref{Glem3def1}, we obtain
\begingroup
\allowdisplaybreaks
\begin{align}\nonumber
&\left|\widehat{B}_{2m+1}\right|=\left|\sum_{\ell=0}^{m}A_2(\ell)A_{2m-2\ell}\right|\le \left|A_2(0)\right|\left|A_{2m}\right|+\sum_{\ell=1}^{m-1}\left|A_2(\ell)\right|\cdot\left|A_{2m-2\ell}\right|+\left|A_2(m)\right|A_0\\\nonumber
&=2\pi\left|A_{2m}\right|+\sum_{\ell=1}^{m-1}\left|A_2(\ell)\right|\cdot\left|A_{2m-2\ell}\right|+\left|A_2(m)\right|\\\nonumber
&\hspace{3.5 cm} \left(\text{by}\ \eqref{Glem1def2}\ \text{with}\ k\mapsto 0, \eqref{Glem3def1}, \text{and}\ \eqref{eqn2.13}\ \text{with}\ \left(\a,m\right)\mapsto \left(\frac{27}{4},0\right)\right)\\\nonumber
&\le 2\pi\left(\frac{27}{4}\right)^m+2\sqrt{\pi}\cosh(4\pi)\sum_{\ell=1}^{m-1}\frac{\left(\frac{27}{4}\right)^{m-\ell}}{\ell^{\frac 32}}+\frac{2\sqrt{\pi}\cosh(4\pi)}{m^{\frac 32}}\\\nonumber
&\hspace{8 cm} \left(\text{by}\ \eqref{eqn2.11}\ \text{and}\ \eqref{Glem2eqn2}\ \text{with}\ \a\mapsto\frac{27}{4}\right)\\\nonumber
&=\left(\frac{27}{4}\right)^m\left(2\pi+2\sqrt{\pi}\cosh(4\pi)\sum_{\ell=1}^{m}\frac{\left(\frac{4}{27}\right)^{\ell}}{\ell^{\frac 32}}\right)\le \left(\frac{27}{4}\right)^m\left(2\pi+2\sqrt{\pi}\cosh(4\pi)\sum_{\ell\ge 1}\left(\frac{4}{27}\right)^{\ell}\right)\\\label{Bhatoddestim}
&=\left(\frac{27}{4}\right)^m\left(2\pi+\frac{8\sqrt{\pi}\cosh(4\pi)}{23}\right).
\end{align}
\endgroup 
From \eqref{Bhatevenestim} and \eqref{Bhatoddestim}, it follows that for $m\in \mathbb{N}_0$,
\begin{equation}\label{Bhatfinalestim}
\left|\widehat{B}_m\right|\le \left(\frac{27}{4}\right)^{\frac m2}\left(2\pi+\frac{8\sqrt{\pi}\cosh(4\pi)}{23}\right)\ \ \left(\text{since}\ \widehat{B}_0=1\ \text{by}\ \eqref{Glem3def1}, \eqref{eqn2.13},\ \text{and}\ \eqref{Glem1def1}\right).
\end{equation}
Next, using \eqref{Glem4def1}, we bound $C_{2m}$ for $m\in \mathbb{N}$ as 
\begingroup
\allowdisplaybreaks
\begin{align}\nonumber
\left|C_{2m}\right|&\le \sum_{k=0}^{m}\left|\binom{-k}{m-k}\right|\frac{\left|a_k(13)\right|}{(4\pi)^{2k}}\le \sum_{k=0}^{m}\binom{m-1}{k-1}\left(\frac{3^{12}}{2^{10}\pi^2}\right)^k(2k)!\ \ \left(\text{by}\ \eqref{akestim}\right)\\\label{Cevenestim}
&\le (2m)!\sum_{k=0}^{m}\binom{m-1}{k-1}53^k\le (2m)!\cdot 54^m,
\end{align}
\endgroup
and in a similar fashion, using \eqref{Glem4def1} and \eqref{akestim}, we get for $m\in \mathbb{N}$,
\begin{equation}\label{Coddestim}
\left|C_{2m+1}\right|\le 8\cdot  (2m+1)!\cdot  69^m.
\end{equation}
From \eqref{Cevenestim}, \eqref{Coddestim}, and using the fact that $C_0=1$ (by \eqref{Glem4def1} and definition of $a_k(13)$ in Theorem \ref{thmKB}), it follows that
\begin{equation}\label{Cfinalestim}
\left|C_m\right|\le 8\cdot m!\cdot 69^{\frac m2}.
\end{equation}
Following the definition of $\widetilde{E}_1(n,N)$ in \eqref{thmeqn1}, we estimate an upper bound in the following way
\begingroup
\allowdisplaybreaks
\begin{align}\nonumber
&\left|\widetilde{E}_1(n,N)\right|\le n^{-\frac{N+1}{2}}\sum_{m=0}^{N-1}\frac{1}{n^{\frac m2}}\sum_{k=m}^{N-1}\left|\widehat{B}_{k+1}\right|\cdot \left|C_{N+m-k}\right|\\\nonumber
&\le 8\left(2\pi+\frac{8\sqrt{\pi}\cosh(4\pi)}{23}\right)n^{-\frac{N+1}{2}}\sum_{m=0}^{N-1}\frac{1}{n^{\frac m2}}\sum_{k=m}^{N-1}\left(\frac{27}{4}\right)^{\frac{k+1}{2}}69^{\frac{N+m-k}{2}}(N+m-k)!\\\nonumber
&\hspace{10 cm}\ \left(\text{by}\ \eqref{Bhatfinalestim}\ \text{and}\ \eqref{Cfinalestim}\right)\\\nonumber
&\le 8\left(2\pi+\frac{8\sqrt{\pi}\cosh(4\pi)}{23}\right) N!\cdot 69^{\frac{N}{2}}n^{-\frac{N+1}{2}}\sum_{m=0}^{N-1}\left(\frac{69}{n}\right)^{\frac m2}\sum_{k=m}^{N-1}\left(\frac{27}{4}\right)^{\frac{k+1}{2}}69^{-\frac{k}{2}}\\\nonumber
&=54\left(2\pi+\frac{8\sqrt{\pi}\cosh(4\pi)}{23}\right) N!\cdot 69^{\frac{N}{2}}n^{-\frac{N+1}{2}}\sum_{m=0}^{N-1}\left(\frac{69}{n}\right)^{\frac m2}\sum_{k=m}^{N-1}\left(\frac{27}{4\cdot 69}\right)^{\frac{k}{2}}\\\nonumber
&\le 54\left(2\pi+\frac{8\sqrt{\pi}\cosh(4\pi)}{23}\right) N!\cdot 69^{\frac{N}{2}}n^{-\frac{N+1}{2}}\sum_{m=0}^{N-1}\left(\frac{27}{4n}\right)^{\frac m2}\sum_{k\ge 0}\left(\frac{27}{4\cdot 69}\right)^{\frac{k}{2}}\\\nonumber
&\le 54\left(2\pi+\frac{8\sqrt{\pi}\cosh(4\pi)}{23}\right) N!\cdot 69^{\frac{N}{2}}n^{-\frac{N+1}{2}}\sum_{m=0}^{N-1}\left(\frac{27}{4n}\right)^{\frac m2}\sum_{k\ge 0}\left(\frac 13\right)^k\\\nonumber
&=81\left(2\pi+\frac{8\sqrt{\pi}\cosh(4\pi)}{23}\right) N!\cdot 69^{\frac{N}{2}}n^{-\frac{N+1}{2}}\sum_{m=0}^{N-1}\left(\frac{27}{4n}\right)^{\frac m2}\\\nonumber
&\le 81\left(2\pi+\frac{8\sqrt{\pi}\cosh(4\pi)}{23}\right) N!\cdot 69^{\frac{N}{2}}n^{-\frac{N+1}{2}}\sum_{m\ge 0}\left(\frac{27}{4\cdot 138}\right)^{\frac m2}\ \ \left(\text{since}\ n\ge 138\right)\\\nonumber
&\le 108\left(2\pi+\frac{8\sqrt{\pi}\cosh(4\pi)}{23}\right) N!\cdot 69^{\frac{N}{2}}\cdot n^{-\frac{N+1}{2}}=\frac{216}{17\cdot 69^{\frac 12}}\left(\pi+\frac{\sqrt{2}\cdot E_2}{23}\right)E_5(N)\cdot n^{-\frac{N+1}{2}}\\\nonumber
&\hspace{4 cm} \left(\text{using the definitions of}\ E_2\ \text{and}\ E_5(N)\ \text{from Lemmas \ref{Glem1}\ and\ \ref{Glem4}} \right)\\\label{errorpart1}
&\le \left(5\cdot E_5(N)+\frac{E_2\cdot E_5(N)}{10}\right)n^{-\frac{N+1}{2}}.
\end{align}
\endgroup
Using \eqref{thmeqn1} and \eqref{Bhatfinalestim}, we obtain
\begingroup
\allowdisplaybreaks
\begin{align}\nonumber
\left|\widetilde{E}_2(n,N)\right|&\le E_6(N)\cdot n^{-\frac{N+1}{2}}\sum_{m=0}^{N}\frac{\left|\widehat{B}_m\right|}{n^{\frac m2}}\\\nonumber
&\le E_6(N)\cdot n^{-\frac{N+1}{2}} \left(2\pi+\frac{8\sqrt{\pi}\cosh(4\pi)}{23}\right)\sum_{m=0}^{N}\left(\frac{27}{4\cdot n}\right)^{\frac m2}\\\nonumber
&\le E_6(N)\cdot n^{-\frac{N+1}{2}}\left(2\pi+\frac{8\sqrt{\pi}\cosh(4\pi)}{23}\right)\sum_{m\ge 0}\left(\frac{27}{4\cdot 138}\right)^{\frac m2}\ \ \left(\text{as}\ n\ge 138\right)\\\label{errorpart2}
&\le \frac 43\left(2\pi+\frac{8\sqrt{\pi}\cosh(4\pi)}{23}\right)E_6(N)\cdot n^{-\frac{N+1}{2}}\le \left(9+\frac{E_2}{6}\right)E_6(N)\cdot n^{-\frac{N+1}{2}}.
\end{align}
\endgroup
For the remaining sum $\widetilde{E}_3(n,N)$, applying \eqref{thmeqn1} and \eqref{Cfinalestim}, it follows that
\begingroup
\allowdisplaybreaks
\begin{align}\nonumber
\left|\widetilde{E}_3(n,N)\right|&\le E_4(N)\cdot n^{-\frac{N+1}{2}}\sum_{m=0}^{N}\frac{\left|C_m\right|}{n^{\frac m2}}\le 8\cdot E_4(N)\cdot n^{-\frac{N+1}{2}} \sum_{m=0}^{N}\left(\frac{69}{n}\right)^{\frac m2}m!\\\label{errorpart3}
&\le 8\cdot E_6(N)\cdot N!\cdot n^{-\frac{N+1}{2}}\sum_{m\ge 0}\left(\frac{1}{2}\right)^{\frac m2}\ \ \left(\text{as}\ n\ge 138\right)\le 28\cdot E_4(N)\cdot N!\cdot n^{-\frac{N+1}{2}}.
\end{align}
\endgroup 
Setting 
\begin{equation}\label{finalerrordef}
E_N:=E_2\frac{3\cdot E_5(N)+5\cdot E_6(N)}{30}+E_4(N)\left(28\cdot N!+E_6(N)\right)+5\cdot E_5(N)+9\cdot E_6(N),
\end{equation}
and applying \eqref{errorpart1}-\eqref{errorpart3} to \eqref{thmeqn1}, we finish the proof of Theorem \ref{thm1.1}.
\qed 

\section{Conclusion}\label{conclusion}
Based on our main result Theorem \ref{thm1.1} and its applications to the combinatorial inequalities presented in Corollaries \ref{cor1}-\ref{cor3}, we discuss both the advantages and disadvantages of our framework. First, we point out the benefits.
\begin{enumerate}
\item It is evident that our method can be applied to get an error estimate for the asymptotic expansion of the coefficients, say $c_f(n)$, arising from the Fourier expansion of a modular form $f$  of negative weight starting from the Rademacher type exact formula (cf. \cite[Theorem 1.1]{Chern}, \cite[Theorem 1.1]{Sussman}, \cite[Theorem 1]{Z})  and applying the error estimate of the $I$-Bessel function given in \cite[Theorems 3.9, 4.7, 5.10]{KB}.
\item Looking from the applications of asymptotic expansions for the coefficients $(c_f(n)$, to prove the combinatorial inequalities for $c_f(n)$ with explicit cutoff for $n$, it would be fair enough to say that in the exact formula of $c_f(n)$, when the order of $I$-Bessel function is small enough, then following the framework we have developed here, can be useful to prove log-concavity, Tur\'{a}n inequality of order $3$, Laguerre inequalities in a unified way. 
To illustrate, for example, we can choose $c_f(n)$ to be the distinct partition function, Broken $k$-diamond partition function (with $k=\{1,2\}$), $k$-regular overpartitions among many others where order of $I$-Bessel functions are either $1$ or $2$ in their Rademacher type exact formula. Unlike proving individually the combinatorial inequalities for such $c_f(n)$ mentioned above as done in \cite{DJ,DJJ,Jia,PZZ,Yang}, first we need to get an estimate for the error bound of the asymptotic expansion of $c_f(n)$ following the proof of Theorem \ref{thm1.1} and then that of for $c_f(n+s)$ for $s\in \mathbb{Z}$. The final stage will be just to choose an appropriate truncation point $N$ to prove such inequalities explicitly.
\end{enumerate}
Finally, we conclude this paper by narrating down the disadvantages from the perspective of determining inequalities with explicit cut off.
\begin{enumerate}
\item With respect to the discussion noted in the second point above, whenever order of the $I$-Bessel function, say $\nu(f)$, is small enough, growth of the coefficients $a_{\nu(f)}(m)=\frac{\binom{\nu(f)-\frac 12}{m}\left(\nu(f)+\frac 12\right)_m}{2^m}$ (cf. Theorem \ref{thmKB}) in the asymptotic expansion of the $I$-Bessel function is within control so as to prove the inequalities for $c_f(n)$ in a definite way. But when both the order $\nu(f)$ and the truncation point $N$ increase, the error bound in the asymptotic expansion will be of exponential growth which seems to precludes deciding from which point on, the inequalities hold for $c_f(n)$ and even if it is decided, verifying the remaining cases would be a difficult task. For instance, in our present case, for $\nu(f)=13$, $N=9$ and even for this, our method is not worthy enough to decide the cutoff of inequalities for $c_f(n)=p_{24}(n)$. This in turn immediately raises the following issues.
\begin{enumerate}
\item Although, we have explicitly computed the error bound $E_N$ for the asymptotic expansion of $p_{24}(n)$, but it is not the optimal one. So, a natural question would be about the growth of the optimal error bound. Because then we can explain about the obstructions to prove inequalities for $p_{24}(n)$ explicitly. 
\item  Using Theorem \ref{thm1.1}, whether would it be possible to prove that the Jensen polynomials $J^{d,n}_{p_{24}}(x)$ has all real roots for $n\ge N(d)$ with an explicit estimate of $N(d)$ so that we can verify the remaining cases at least for $d=\{2,3,4\}$ by staying within the framework developed by Griffin, Ono, Rolen, and Zagier in \cite{GORZ} instead of translating the problem into the question of positivity of the discriminants of respective Jensen polynomials?    
\end{enumerate} 
\end{enumerate}

		\begin{center}
			\textbf{Acknowledgements}
		\end{center}
The author expresses her sincere gratitude to Prof. Brundaban Sahu for his valuable advice and suggestions on the problem. The author would also like to thank the institute for its hospitality and support. 



	\end{document}